\documentclass[11pt]{amsart}
\usepackage{amsmath}
\usepackage{amsfonts}
\usepackage{amssymb}
\usepackage{latexsym}

\newtheorem{theorem}{Theorem}[section]
\newtheorem{lemma}[theorem]{Lemma}
\newtheorem{coro}[theorem]{Corollary}
\newtheorem{proposition}[theorem]{Proposition}
\newtheorem{otherth}{\bf Theorem}

\newtheorem{otherl}{\bf Lemma}

\newcommand{\Cn}{\mathbb{C}^n}
\newcommand{\Bn}{\mathbb{B}_ n}
\newcommand{\Sn}{\mathbb{S}_ n}

\numberwithin{equation}{section}

\begin{document}
\title[Integration operators on Hardy spaces]{Integration operators between Hardy spaces on the unit ball of $\Cn$}
\
\author[Jordi Pau]{Jordi Pau}
\address{Jordi Pau \\Departament de Matem\`{a}tica Aplicada i Analisi\\
Universitat de Barcelona\\ Gran Via 585 \\
08007 Barcelona\\
Catalonia} \email{jordi.pau@ub.edu}



\thanks{The author is
 supported by SGR grant $2009$SGR $420$ (Generalitat de
Catalunya) and DGICYT grant MTM$2011$-$27932$-$C02$-$01$
(MCyT/MEC)}

\maketitle

\begin{abstract}
We completely describe the boundedness of the Volterra type operator $J_ g$ between Hardy spaces in the unit ball of $\Cn$. The proof of the one dimensional case used tools, such as the strong factorization for Hardy spaces, that are not available in higher dimensions, and therefore new techniques are developed. In particular, a generalized version of the description of Hardy spaces in terms of the area function is needed.
\end{abstract}

\section{Introduction and main results}
Let $\Bn$ be the open unit ball in $\Cn$. Denote by $H(\Bn)$ the space of all holomorphic functions in $\Bn$. For a function $g\in H(\Bn)$,
define the operator
\begin{equation}\label{defJ}
 J_ g f(z)=\int_{0}^{1} f(tz) Rg(tz) \frac{dt}{t},\qquad z\in \Bn
\end{equation}
for $f$ holomorphic in $\Bn$. Here $Rg$ denotes the radial derivative of $g$, that is,
\[Rg(z)= \sum_{k=1}^{n} z_ k \frac{\partial g}{\partial z_ k} (z),\qquad z=(z_ 1,\dots,z_ n)\in \Bn.\]
In the one dimensional case $n=1$, the operator $J_ g$ was first considered in the setting of Hardy spaces by Pommerenke \cite{Pom} related to the study of certain properties of $BMOA$ functions. We want to mention here that a closely related operator was introduced earlier by Calder\'{o}n in \cite{Cal}. After the pioneering works of Aleman, Siskakis and Cima \cite{AC,AS0,AS} describing the boundedness and compactness of the operator $J_ g$ in Hardy and Bergman spaces, the mentioned operator became extremely popular, being studied in many spaces of analytic functions (see \cite{AC,AO,AS0,AS,OC,PP, PR} for example). As far as we know, the generalization of the operator $J_ g$ acting on holomorphic functions in the unit ball of $\Cn$ (as defined here) was introduced by Z. Hu \cite{Hu}. A fundamental property of the operator $J_ g$, that follows from an easy calculation with \eqref{defJ}, is the following basic formula involving the radial derivative $R$ and the operator $J_ g$:
\begin{equation}\label{form}
R(J_ g f)(z)=f(z)\,Rg(z),\quad z\in \Bn.
\end{equation}
 The boundedness and compactness of $J_ g$ has been extensively studied in many spaces of holomorphic functions in the unit ball (see \cite{XiaoLMS} and \cite{Xiao-Conf} for the corresponding study on Bergman and Bloch type spaces). However, the case of the Hardy spaces on the unit ball, that is, the study of $J_ g:H^p(\Bn)\rightarrow H^q(\Bn)$ (that, in my opinion, is the most important case, and is the setting were the operator $J_ g$ was originally studied) is missing, only the elementary case $q=p=2$ (see \cite{LS1}) and the case $p<q$ (see \cite{AS}) has been done before. Our goal is to fill this gap, and we completely describe the boundedness and compactness of $J_ g:H^p(\Bn)\rightarrow H^q(\Bn)$ for all $0<p,q<\infty$.

 For $0<p<\infty$, the Hardy space $H^p:=H^p(\Bn)$ consists of those holomorphic functions $f$ in $\Bn$ with
 \[ \|f\|_{H^p}^p=\sup_{0<r<1}\int_{\Sn} \!\! |f(r\zeta)|^p \,d\sigma(\zeta)<\infty,\]
where $d\sigma$ is the surface measure on the unit sphere $\Sn:=\partial \Bn$ normalized so that $\sigma(\Sn)=1$. We refer to the books \cite{Alek}, \cite{Rud} and \cite{ZhuBn} for the theory of Hardy spaces in the unit ball.

 The norm of the operator $J_ g:H^p\rightarrow H^q$ is denoted by $\|J_ g\|_{H^p\rightarrow H^q}$ and, when $q=p$  its norm is simply denoted by $\|J_ g\|$. Now we are ready to state our main results describing the boundedness of $J_ g:H^p\rightarrow H^q$ extending the one-dimensional results obtained by Aleman-Siskakis \cite{AS0} (the case $q=p\ge 1$) and by Aleman-Cima \cite{AC} (the remainder cases).

\begin{theorem}\label{mt1}
Let $0<p<\infty$ and $g\in H(\Bn)$. Then
$J_ g$ is bounded on $H^p$ if and only if $g\in BMOA$. Moreover,
\[ \|J_ g\|\asymp \|g\|_{BMOA}.\]
\end{theorem}
Here, the notation $A\asymp B$ means that the two quantities are comparable. We want to mention here that, in one dimension, a different proof (of some parts) of that in \cite{AS}, \cite{AC} has been given recently in \cite{PR} and \cite{WuIEOT}. In my opinion, the proof we will give here (of course valid also in one dimension) is more simpler and elegant than the ones presented before.

In order to state the case $p<q$ we need to introduce the Lipschitz type spaces $\Lambda(\alpha)$. For $0<\alpha\le 1$, we say that an analytic function $g$ belongs to the Lipschitz type space $\Lambda(\alpha)$ if
\[\|g\|_{\Lambda(\alpha)}=\sup_{z\in \Bn} (1-|z|^2)^{1-\alpha} \,|Rg(z)|<\infty.\]
This coincides \cite[Chapter 7]{ZhuBn} with the space of holomorphic functions $g$ in $\Bn$ with
\[|g(z)-g(w)|\le C |z-w|^{\alpha},\qquad z,w\in \Bn.\]
\begin{theorem}\label{mt2}
Let $0<p<q<\infty$, $g\in H(\Bn)$ and $\alpha=n(\frac{1}{p}-\frac{1}{q})$.
 \begin{enumerate}
 \item[(a)] If $\alpha\le 1$ then $J_ g:H^p\rightarrow H^q$ is bounded if and only if $g\in \Lambda(\alpha)$. Moreover,
\[\|J_ g\|_{H^p\rightarrow H^q}\asymp \|g\|_{\Lambda(\alpha)}.\]
\item[(b)] If $\alpha>1$, then  $J_ g:H^p\rightarrow H^q$ is bounded if and only if $g$ is constant, that is, $J_ g\equiv 0$.
\end{enumerate}
\end{theorem}
After the finishing of the paper we realized that Theorem \ref{mt2} has been also obtained recently in \cite{AS}. For completeness and convenience of the reader, we offer our proof here. It remains to deal with the other non diagonal case, result that is stated below.

\begin{theorem}\label{mt3}
Let $0<q<p<\infty$ and $g\in H(\Bn)$. Then $J_ g :H^p\rightarrow H^q$ is bounded if and only if $g\in H^r$, where $\frac{1}{r}=\frac{1}{q}-\frac{1}{p}$. Moreover, we have
\[ \|g\|_{H^r}\asymp \|J_ g\|_{H^p\rightarrow H^q}.\]
\end{theorem}
The proofs of the previous results in the one dimensional setting used, in a decisive way, tools such as the strong factorization for Hardy spaces and some results of Aleksandrov and Peller \cite{AP} that are not available in higher dimensions, so that the generalization to the unit ball of $\Cn$ is not a routine that any machine can do, and new techniques and ideas must be developed. We also want to notice that in the proofs of the previous theorems we can always assume that $g(0)=0$ since $J_ g=J_{g+c}$ for any constant $c$.

The paper is organized as follows: in Section \ref{s2} we recall some well known results that will be used in the proofs. Theorems \ref{mt1}, \ref{mt2} and \ref{mt3} are proved in Sections \ref{s3}, \ref{s4} and \ref{s5} respectively. Characterizations of the compactness of the integration operator $J_ g$ and membership in the Schatten-Von Neumann ideals $S_ p(H^2)$ are obtained in Section \ref{s6}.

Throughout the paper, constants are
often given without computing their exact values, and the value of a constant $C$ may change
from one occurrence to the next. We also use the notation $a\lesssim b$ to indicate that there is a constant $C>0$
with $a\le C b$.

\section{Background}\label{s2}
In this section we introduce some notation and recall some well known results that will be used throughout the paper. For any two points $z=(z_ 1,\dots,z_ n)$ and $w=(w_ 1,\dots,w_ n)$ in $\Cn$ we write
$$\langle z,w\rangle =z_ 1\bar{w}_ 1+\dots +z_ n \bar{w}_ n,$$
 and
 $|z|=\sqrt{\langle z,z\rangle}=\sqrt{|z_ 1|^2+\dots +|z_ n|^2}.$
 Denote by $dv$ the usual Lebesgue volume measure on $\mathbb{B}_ n$, normalized so that the volume of $\mathbb{B}_ n$ is one.
\subsection{Invariant type derivatives}
Let
\begin{displaymath}
\Delta=4 \sum_{k=1}^{n} \frac{\partial ^2}{\partial z_ k \,\partial \bar{z}_ k} =\sum_{k=1}^{n} \left (\frac{\partial ^2}{\partial x_ k^2}+\frac{\partial ^2}{\partial y_ k^2}\right )
\end{displaymath}
be the standard Laplace operator on $\Cn$, where
\begin{displaymath}
\frac{\partial}{\partial z_ k}=\frac{1}{2} \left (\frac{\partial}{\partial x_ k}-i \frac{\partial}{\partial y_ k}\right ) \quad \textrm{and}\quad \frac{\partial}{\partial \bar{z}_ k}=\frac{1}{2} \left (\frac{\partial}{\partial x_ k}+i \frac{\partial}{\partial y_ k}\right )
\end{displaymath}
provided the use of the identification $z_ k=x_ k+i y_ k$ for $1\le k\le n$ is made. If $f$ is a twice differentiable function in $\Bn$, the invariant Laplacian of $f$ is defined as
\[ (\widetilde{\Delta} f )(z)=\Delta(f\circ \varphi_ z)(0),\qquad z\in \Bn, \]
where $\varphi_ z$ is the automorphism of $\Bn$ that interchanges the points $0$ and $z$.

If $f$ is a differentiable function in $\Bn$, we use $\nabla f$ to denote its real gradient. The (real) invariant gradient of $f$ is then defined as
\[\widetilde{\nabla} f(z)= \nabla (f\circ \varphi_ z)(0),\qquad z\in \Bn.\]
When $f$ is holomorphic on $\Bn$ it is typical to use also the complex gradient
\[\nabla _ h f(z)=\left (\frac{\partial f}{\partial z_ 1} (z) , \dots,\frac{\partial f}{\partial z_ n} (z) \right )\]
and call $|\nabla _ h f(z)|$ the holomorphic gradient of $f$ at $z$. Similarly, one defines
$\widetilde{\nabla}_ h f(z)= \nabla _ h(f\circ \varphi_ z)(0)$, $z\in \Bn$,
and refer to the quantity $|\widetilde{\nabla}_ h f(z)|$ as the holomorphic invariant gradient of $f$ at $z$. This can not create any confusion, since for $f$ holomorphic, one has $|\nabla f(z)|=2 \,|\nabla _ h f(z)|$.

\subsection{The invariant Green's formula}
It is a consequence of the invariant Green's formula \cite[Theorem 1.25]{ZhuBn} that, if $f$ is of class $C^2$ on $\Bn$ then
\begin{displaymath}
\int_{\Bn} \widetilde{\Delta} f(z)\, G(z)\,d\lambda_ n(z)=\int_{\Sn} f(\zeta)\,d\sigma(\zeta)-f(0),
\end{displaymath}
(see \cite{PTW}) where $G(z)$ is the invariant Green function of $\Bn$ given by
\[ G(z)=\frac{1}{2n} \int_{|z|}^{1} (1-t^2)^{n-1} t^{-2n+1} dt, \]
and
 $$d\lambda_ n(z)=\frac{dv(z)}{(1-|z|^2)^{n+1}}$$
  is the hyperbolic or invariant  measure on $\Bn$. The constant appearing in \cite{PTW} is absorbed in the normalized measure $dv$ since the volume of $\Bn$ is exactly $\pi^n/ n !$.

\subsection{Hardy-Stein type inequalities}
It is a consequence of the Hardy-Stein identity for the ball (see \cite[Chapter 4]{ZhuBn} or \cite{Now}) that, if $g(0)=0$, then for $0<p<\infty$ one has
\begin{displaymath}
\|g\|^p_{H^p}\asymp \int_{\Bn} |g(z)|^{p-2}\,|Rg(z)|^2 (1-|z|^2)\,dv(z).
\end{displaymath}
There are analogues of these inequalities using the gradient or the invariant gradient instead of the radial derivative \cite{ZhuBn}, \cite{Stoll1}. For example, in terms of the gradient, one simply replaces $Rg$ in the above estimate by the real gradient $\nabla g$, and using the invariant gradient, one has the following:
\begin{displaymath}
\|g\|^p_{H^p}\asymp \int_{\Bn} |g(z)|^{p-2}\,|\widetilde{\nabla} g(z)|^2 (1-|z|^2)^n\,d\lambda_ n(z).
\end{displaymath}
Given a function $f\in L^1(\Sn)$, the invariant Poisson integral of $f$, denoted by $u_ f$, is defined on $\Bn$ as
\[ u_ f(z)=\int_{\Bn} f(\zeta)\,\frac{(1-|z|^2)^n}{|1-\langle z,\zeta \rangle |^{2n}}\,d\sigma(\zeta).\]
Note that the invariant Poisson kernel here is different from the associated Poisson kernel when $\Bn$ is thought of as the unit ball in $\mathbb{R}^{2n}$, unless $n=1$. The invariant Poisson integral $u_ f$ is $\mathcal{M}$-harmonic on $\Bn$, meaning that is annihilated by the invariant Laplacian, that is, $\widetilde {\Delta} u_ f=0$ . The version of the Hardy-Stein inequalities for $\mathcal{M}$-harmonic functions
(see \cite{JP} or \cite{Stoll1}) is the following: let $1<p<\infty$ and $f\in L^p(\Sn)$. Then
\[ \|f\|^p_{L^p(\Sn)}\asymp |u_ f(0)|^2+\int_{\Bn} |u_ f(z)|^{p-2} |\widetilde{\nabla} u_ f(z)|^2 \,(1-|z|^2)^n \,d\lambda_ n(z).\]
When $p=2$ the previous estimates are usually referred as the Littlewood-Paley inequalities.
\subsection{Admissible maximal and area functions}
For $\zeta \in \Sn$ and $\alpha>1$ the admissible approach region $\Gamma_{\alpha}(\zeta)$ is defined as
\begin{displaymath}
\Gamma(\zeta)=\Gamma_{\alpha}(\zeta)=\left \{z\in \Bn: |1-\langle z,\zeta\rangle |<\frac{\alpha}{2} (1-|z|^2) \right \}.
\end{displaymath}
If $I(z)=\{\zeta \in \Sn: z\in \Gamma(\zeta)\}$, then $\sigma(I(z))\asymp (1-|z|^2)^{n}$, and it follows from Fubini's theorem that, for a positive function $\varphi$, and a finite positive measure $\nu$, one has
\begin{equation}\label{EqG}
\int_{\Bn} \varphi(z)\,d\nu(z)\asymp \int_{\Sn} \left (\int_{\Gamma(\zeta)} \varphi(z) \frac{d\nu(z)}{(1-|z|^2)^{n}} \right )d\sigma(\zeta).
\end{equation}
This fact will be used repeatedly throughout the paper. \\

For $\alpha>1$ and $f$ continuous on $\Bn$, the admissible maximal function $f^*_{\alpha}$ is defined on $\Sn$ by
\[ f^*(\zeta)=f_{\alpha}^*(\zeta)=\sup_{z\in \Gamma_{\alpha}(\zeta)} |f(z)|.\]
We need the following well known result on the $L^p$-boundedness of the admissible maximal function that can be found in \cite[Theorem 5.6.5]{Rud} or \cite[Theorem 4.24]{ZhuBn}.
\begin{otherth}\label{NTMT}
Let $0<p<\infty$ and $f\in H(\Bn)$. Then
\[ \|f^*\|_{L^p(\Sn)}\le C \|f\|_{H^p}.\]
\end{otherth}
Another function we need is the admissible area function $A_{\alpha}f$ defined on $\Sn$ by
\[Af(\zeta)=A_{\alpha}f(\zeta)=\left ( \int_{\Gamma_{\alpha}(\zeta)} |Rf(z)|^2 \,(1-|z|^2)^{1-n}dv(z)\right )^{1/2}.\]
The following result \cite{AB}, \cite{FS} describing the functions in the Hardy space in terms of the admissible area function, is the version for the unit ball of $\Cn$ of the famous Calder\'{o}n area theorem \cite{Cal} who extended to all $0<p<\infty$ the result proved for $p>1$ by Marcinkiewicz and Zygmund \cite{MZ}.
\begin{otherth}\label{AreaT}
Let $0<p<\infty$ and $g\in H(\Bn)$. Then $g\in H^p$ if and only if $Ag\in L^p(\Sn)$. Moreover, if $g(0)=0$ then
\[ \|g\|_{H^p}\asymp \|Ag\|_{L^p(\Sn)}.\]
\end{otherth}
A generalized version of Theorem \ref{AreaT} is given in Theorem \ref{KP}, with a proof that includes Theorem \ref{AreaT} itself.

\subsection{Embedding of Hardy spaces into Bergman spaces}
For $0<p<\infty$ and $\alpha>-1$, the weighted Bergman space $A^p_{\alpha}(\Bn)$ consists of those functions $f$ holomorphic on $\Bn$ with
\[ \|f\|_{A^p_{\alpha}}=\left (\int_{\Bn} |f(z)|^p\,dv_{\alpha}(z) \right )^{1/p}<\infty.\]
Here
$dv_{\alpha}(z)=c_{\alpha}\,(1-|z|^2)^{\alpha}dv(z),$
where $c_{\alpha}$ is a positive constant chosen so that $v_{\alpha}(\Bn)=1$.
We will make use of the following result that appears in \cite[Theorem 4.48]{ZhuBn}.
\begin{otherth}\label{Hardy-Bergman}
For $0<p<q<\infty$ we have $H^p\subset A^q_{\alpha}(\Bn)$ with
$$\alpha=nq \Big (\frac{1}{p}-\frac{1}{q}\Big )-1=\frac{nq}{p}-(n+1).$$
Moreover, there exists a constant $C>0$ such that $\|f\|_{A^q_{\alpha}} \le C \|f\|_{H^p}.$
\end{otherth}

\subsection{Carleson measures and $\mathbf{BMOA}$}
For $\zeta \in \Sn$ and $\delta>0$ consider the sets
\[B_{\delta}(\zeta)=\big \{ z\in \Bn: |1-\langle z, \zeta \rangle |<\delta \big \}.\]
A positive Borel measure $\mu$ on $\Bn$ is said to be a Carleson measure if there exists a constant $C>0$ such that
\[\mu \big (B_{\delta}(\zeta)\big )\le C \delta \, ^ n\]
for all $\zeta \in \Sn$ and $\delta>0$. Obviously every Carleson measure is finite. H\"{o}rmander \cite{H} extended to several complex variables the famous Carleson measure theorem \cite{Car0, Car1} by proving that, for $0<p<\infty$, the embedding $I_ d:H^p\rightarrow L^p(\mu):=L^p(\Bn,d\mu)$ is bounded if and only if $\mu$ is a Carleson measure.\\

The space of analytic functions of bounded mean oscillation $BMOA=BMOA(\Bn)$ consists of those functions $f\in H^1$ with
\begin{displaymath}
\|f\|_{BMOA}=|f(0)|+\sup \frac{1}{\sigma(Q)}\int_{Q} \! |f(\zeta)-f_ Q|\,d\sigma(\zeta)<\infty,
\end{displaymath}
where $f_ Q=\frac{1}{\sigma(Q)}\int_{Q} f \,d\sigma$ is the mean of $f$ over $Q$ and the supremum is taken over the non-isotropic metric balls $Q=Q(\zeta,\delta)=\{ \xi\in \Sn : |1-\langle \zeta, \xi \rangle |<\delta \}$ for all $\zeta \in \Sn$ and $\delta>0$. The next result \cite[Chapter 5]{ZhuBn} gives an alternate description of $BMOA$ in terms of Carleson measures.
\begin{otherth}
Let $g\in H(\Bn)$ and consider the measure $\mu_ g$ defined by
\[d\mu_ g(z)=|Rg(z)|^2 (1-|z|^2)\,dv(z).\]
Then $g\in BMOA$ if and only if $\mu_ g$ is a Carleson measure. Moreover, if $g(0)=0$, for all $0<p<\infty$ one has
\begin{equation}\label{bmoa}
\|g\|_{BMOA}\asymp \sup_{\|f\|_{H^p}=1} \left (\int_{\Bn} |f(z)|^p \,d\mu_ g(z) \right )^{1/2}.
\end{equation}
\end{otherth}
We also will need the following result essentially due to Luecking \cite{Lue1}. Since  Luecking result is stated for real Hardy spaces, for convenience of the reader, and in order to offer no doubt of the validity of the result, we give a proof at the end of the paper.
\begin{otherth}\label{LT}
Let $0<s<p<\infty$ and let $\mu$ be a positive Borel measure on $\Bn$. Then the identity $I_ d:H^p\rightarrow L^s(\mu)$ is bounded, if and only if, the function defined on $\Sn$ by
\[ \widetilde{\mu}(\zeta)=\int_{\Gamma(\zeta)} (1-|z|^2)^{-n}d\mu(z) \]
belongs to $L^{p/(p-s)}(\Sn)$. Moreover, one has $\|I_d\|_{H^p\rightarrow L^s(\mu)}\asymp \|\widetilde{\mu}\|_{L^{p/(p-s)}(\Sn)}^{1/s}.$
\end{otherth}

\section{Proof of Theorem \ref{mt1}}\label{s3}
Consider the measure $\mu_ g$ defined by
\[ d\mu_ g(z)=|Rg(z)|^2 (1-|z|^2)\,dv(z).\]
The case $p=2$ is particularly simple. Indeed, by the Littlewood-Paley inequalities and the basic formula \eqref{form} one has
\begin{equation}\label{Eqp2}
\|J_ g f\|^2_{H^2} \asymp \int_{\Bn} |f(z)|^2\,d\mu_ g (z) \le C \|f\|_{H^2}^2
\end{equation}
if and only if $g\in BMOA$ with $\|J_ g\|\asymp \|g\|_{BMOA}$ due to \eqref{bmoa}. Now we are going to consider the other cases.
\subsection{Sufficiency}
Suppose  that $g\in BMOA$. We want to prove that
\begin{equation}\label{EqB}
\|J_ g f \|_{H^p}\le C \|g\|_{BMOA}\cdot \|f\|_{H^p}\,.
\end{equation}
By taking $f$ in the ball algebra (the algebra of all holomorphic functions in $\Bn$ continuous up to the boundary, a dense subset of $H^p$), and then using an standard approximation argument, it is enough to establish \eqref{EqB} assuming that $\|J_ g f\|_{H^p}$ is already finite.
For $p\ge 2$, we use the Hardy-Stein inequalities, the basic formula \eqref{form}, H\"{o}lder's inequality and then \eqref{bmoa} to get
\begin{displaymath}
\begin{split}
\|J_ g f\|^p_ {H^p} &\asymp \int_{\Bn} |J_ g f(z)|^{p-2} \,|R(J_ gf)(z)|^2 \,(1-|z|^2)\,dv(z)
\\
&=\int_{\Bn} |J_ g f(z)|^{p-2} \,|f(z)|^2\,|Rg(z)|^2 \,(1-|z|^2)\,dv(z)
\\
&\le \left( \int_{\Bn} |J_ gf(z)|^p d\mu_ g(z)\right )^{\frac{p-2}{p}} \left( \int_{\Bn} |f(z)|^p d\mu_ g(z)\right )^{\frac{2}{p}}
\\
&
\le C \| g\|^2_{BMOA} \cdot \|J_ g f \|_{H^p}^{p-2} \cdot \|f\|_{H^p}^2.
\end{split}
\end{displaymath}
Hence we obtain that
\[ \|J_ g f\|^2_{H^p}\le C \| g\|^2_{BMOA} \cdot \|f\|^2_{H^p},\]
that is, the operator $J_ g$ is bounded on $H^p$ with $\|J_ g\|\le C \| g\|_{BMOA}$.\\
\par
 For $0<p<2$, we use the area function description of $H^p$ (Theorem \ref{AreaT}), the basic identity \eqref{form}, H\"{o}lder's inequality, the $L^p$-boundedness of the admissible maximal function (Theorem \ref{NTMT}), \eqref{EqG} and finally \eqref{bmoa} to get
\begin{displaymath}
\begin{split}
\|J_ g f\|^p_{H^p} &\asymp \|A(J_ g f)\|_{L^p(\Sn)}^p
\\
&= \int _{\Sn} \left (\int_{\Gamma(\zeta)}\!\! |f(z)|^2 |Rg(z)|^2 (1-|z|^2)^{1-n}dv(z)\right )^{p/2} \!\!\!\!d\sigma(\zeta)
\\
&\le \int _{\Sn} (f^*(\zeta))^{\frac{(2-p)p}{2}}\left (\int_{\Gamma(\zeta)}\!\! |f(z)|^p |Rg(z)|^2 (1-|z|^2)^{1-n} dv(z)\right )^{p/2} \!\!\!\!\!d\sigma(\zeta)
\\
&\le \|f^*\|_{L^p(\Sn)}^{\frac{p(2-p)}{2}}\left (\int_{\Sn} \int_{\Gamma(\zeta)} \!\!|f(z)|^p |Rg(z)|^2 (1-|z|^2)^{1-n}dv(z) \,d\sigma(\zeta)\right)^{p/2}
\\
&\le C \|f\|_{H^p}^{\frac{p(2-p)}{2}}\left (\int_{\Bn} |f(z)|^p d\mu_ g(z) \right )^{p/2}
\le C \|g\|_{BMOA}^{p}\cdot  \|f\|^p_{H^p}.
\end{split}
\end{displaymath}
Thus $J_ g$ is bounded on $H^p$ with $\|J_ g\| \le C \| g\|_{BMOA}$.
\subsection{Necessity}
 Suppose now that $J_ g$ is bounded on $H^p$. We consider first the case $p\ge 2$. In this case, \eqref{EqG}, H\"{o}lder's inequality, the $L^p$-boundedness of the admissible maximal function and the area function characterization of $H^p$ functions (Theorem \ref{AreaT}) gives
\begin{displaymath}
\begin{split}
\int_{\Bn} |f(z)|^p d\mu_ g(z) &\asymp\int_{\Sn} \int_{\Gamma(\zeta)} \!|f(z)|^p\,|Rg(z)|^2 (1-|z|^2)^{1-n}dv(z)\, d\sigma(\zeta)
\\
& \le \int_{\Sn} \! (f^*(\zeta))^{p-2} \!\int_{\Gamma(\zeta)}\!\! |f(z)|^2\,|Rg(z)|^2 (1-|z|^2)^{1-n}dv(z) \,d\sigma(\zeta)
\\
&\le \|f^*\|_{L^p(\Sn)}^{p-2} \!\left [\int_{\Sn} \!\!\left (\int_{\Gamma(\zeta)} \!\!|R(J_ g f)(z)|^2\,(1-|z|^2)^{1-n}dv(z)\right )^{p/2} \!\!\!\!\!d\sigma(\zeta)\right ]^{2/p}
\\
& \le C \|f\|_{H^p}^{p-2} \cdot \|J_ g f\|^2_{H^p} \le C \|J_ g\|^2 \cdot \|f\|_{H^p}^{2}.
\end{split}
\end{displaymath}
Taking the supremum over all $f\in H^p$ with $\|f\|_{H^p}=1$ and using \eqref{bmoa}, this shows that $g\in BMOA$ with $\|g\|_{BMOA}\le C\|J_ g\|$.\\

\par Finally, it remains to deal with the case $0<p<2$. By considering the dilated functions $g_{\rho}(z)=g(\rho z)$, $0<\rho<1$, it is enough to prove the inequality
$ \|g\|_{BMOA}\le C \|J_ g\| $
assuming that $g$ is already in $BMOA$. Then a standard limiting argument using that $\lim_{\rho \rightarrow 1^{-}}\|J_{g_{\rho}}\|\lesssim \|J_ g\|$ will give the result. To this end, consider a function $f$ in the Hardy space $H^p$. The use of the Hardy-Stein inequalities together with \eqref{form} yields
\[  \|J_ g f\|^p_{H^p} \asymp \int_{\Bn} \!\! |J_{g} f(z)|^{p-2}\,|f(z)|^2\,d\mu_{g}(z).\]
Now, using H\"{o}lder's inequality, the previous estimate together with \eqref{bmoa} and the boundedness of $J_ g$ on $H^p$, we obtain
\begin{displaymath}
\begin{split}
\int_{\Bn} \!\!\!|f(z)|^p \,d\mu_{g}(z)&\le \left (\int_{\Bn} \!\!\!|J_{g} f(z)|^p\,d\mu_{g}(z)\right )^{1-\frac{p}{2}}
 \! \left (\int_{\Bn} \!\!\!|J_{g} f(z)|^{p-2}\,|f(z)|^2\,d\mu_{g}(z)\right )^{p/2}
 \\
&\le C \Big (\|g\|_{BMOA}^2\cdot \|J_{g}f\|^p_{H^p} \Big)^{1-\frac{p}{2}} \,\|J_{g}f\|_{H^p}^{p^2/2}
\\
& \le C \|g\|_{BMOA}^{2-p} \cdot \,\|J_{g}\|^p \,\cdot \|f\|_{H^p}^{p}.
\end{split}
\end{displaymath}
Taking the supremum over all $f$ with $\|f\|_{H^p}=1$ and using \eqref{bmoa} again gives
\[\|g\|_{BMOA}^2 \le C \,\|g\|_{BMOA}^{2-p} \cdot\,\|J_{g}\|^p. \]
This implies the desired estimate
$\|g\|_{BMOA}\le C \|J_{g}\| $
completing the proof of the Theorem.

\section{Proof of Theorem \ref{mt2}}\label{s4}

\subsection{Necessity} Assume that $J_ g:H^p\rightarrow H^q$ is bounded. The standard estimate for $H^q$ functions gives
$|R(J_ g f)(z)|\le C (1-|z|^2)^{-(n+q)/q}\,\|J_ g f\|_{H^q}.$
It follows from the fundamental identity \eqref{form} that
\begin{displaymath}
|f(z)|\,|Rg(z)| \le C (1-|z|^2)^{-(n+q)/q}\,\|J_ g\|_{H^p\rightarrow H^q}\cdot\|f\|_{H^p}.
\end{displaymath}
Taking the function $f=f_ z$ with
\[ f_ z(w)=\frac{(1-|z|^2)^{n/p}}{(1-\langle w,z\rangle)^{2n/p}}\]
that has $H^p$-norm $1$ we get
\begin{displaymath}
(1-|z|^2)^{-n/p}\,|Rg(z)| \le C (1-|z|^2)^{-(n+q)/q}\,\|J_ g\|_{H^p\rightarrow H^q}.
\end{displaymath}
That is,
$
\|g\|_{\Lambda(\alpha)}\le C \|J_ g\|_{H^p\rightarrow H^q}
$
with $\alpha=n(\frac{1}{p}-\frac{1}{q})$ as desired. This also proves part (b) since, for $\alpha>1$, the condition $(1-|z|^2)^{1-\alpha} |Rg(z)|\le C$ implies that $|Rg(z)|\rightarrow 0$ as $|z|\rightarrow 1^{-}$ and hence $g$ must be constant.

\subsection{Sufficiency}
Let $\alpha=n(\frac{1}{p}-\frac{1}{q})$, and assume that $g\in \Lambda(\alpha)$. We consider first the almost trivial case $q=2$. Here we use the Littlewood-Paley inequalities, the formula \eqref{form} and the embedding of Hardy spaces into Bergman spaces to get
\begin{displaymath}
\begin{split}
\|J_ g f \|^2_{H^2}&\asymp \int_{\Bn}\! \!|f(z)|^2 \,|Rg(z)|^2\,(1-|z|^2)\,dv(z)
\\
& \le \|g\|^2_{\Lambda(\alpha)} \int_{\Bn} \!\!|f(z)|^2 \,(1-|z|^2)^{2\alpha-1}\,dv(z)
 \le C \|g\|^2_{\Lambda(\alpha)} \cdot \|f\|_{H^p}^2,
\end{split}
\end{displaymath}
and this shows that $J_ g:H^p\rightarrow H^2$ is bounded with $\|J_ g\|_{H^p\rightarrow H^2}\le C \|g\|_{\Lambda(\alpha)}$.\\

 Next we deal with the case $q>2$. As noticed in the proof of Theorem \ref{mt1} it is enough to establish the inequality $\|J_ g f\|_{H^q}\le C \|g\|_{\Lambda(\alpha)}\cdot \|f\|_{H^p}$ assuming that $\|J_ g f\|_{H^q}$ is already finite. To this end, take a number $s>q$ with $s<\frac{(q-2)p}{(p-2)}$ if $p>2$ (this choice is possible, since for $p>2$ one has $\frac{(q-2)p}{(p-2)}>q$ due to the fact that $p<q$), and let $\gamma=ns(\frac{1}{p}-\frac{1}{q})$. By the Hardy-Stein inequalities, \eqref{form} and H\"{o}lder's inequality we have
\begin{equation}\label{EqH0}
\begin{split}
\|J_ g f\|_{H^q}^q & \asymp \int_{\Bn} \!\!|J_ g f(z)|^{q-2} |f(z)|^2\,|Rg(z)|^2 (1-|z|^2)dv(z)
\\
&\le C \,\|g\|^2_{\Lambda(\alpha)}\int_{\Bn} \!\!|J_ g f(z)|^{q-2} |f(z)|^2\, (1-|z|^2)^{2\alpha-1}\,dv(z)
\\
&\le C \,\|g\|^2_{\Lambda(\alpha)}\cdot \|J_ g f \|_{A^s_{\gamma-1}}^{q-2} \!\left ( \int_{\Bn}\!\! \!|f(z)|^{\frac{2s}{s-(q-2)}} \, (1-|z|^2)^{\beta-1} dv(z)\right )^{\frac{s-(q-2)}{s}}
\end{split}
\end{equation}
with
$$\beta= \frac{(2\alpha-\gamma)s}{s-(q-2)} +\gamma-1.$$
Since $s>q$, the embedding of Hardy spaces into Bergman spaces (Theorem \ref{Hardy-Bergman}) gives
\begin{equation}\label{Eqs1}
\|J_ g f \|_{A^s_{\gamma-1}} \le C \|J_ g f \|_{H^q}.
\end{equation}
Also, the choice made on the number $s$ ensures that
\begin{displaymath}
 s_ q:=\frac{2s}{s-(q-2)}>p.
\end{displaymath}
Since $\beta=ns_ q \Big (\frac{1}{p}-\frac{1}{s_ q} \Big )$, by making another use of Theorem \ref{Hardy-Bergman}, we have
\begin{equation}\label{Es2}
\int_{\Bn} | f(z)|^{\frac{2s}{s-(q-2)}} \, (1-|z|^2)^{\beta-1}\,dv(z)\le C \|f\|_{H^p}^{\frac{2s}{s-(q-2)}}.
\end{equation}
Putting \eqref{Eqs1} and \eqref{Es2} into \eqref{EqH0} yields
\[ \|J_ g f\|_{H^q}^q\le C \|g\|^2_{\Lambda(\alpha)}\cdot \|J_ g f\|_{H^q}^{q-2} \cdot \|f\|_{H^p}^2 ,\]
that is
\[ \|J_ g f\|_{H^q} \le C \|g\|_{\Lambda(\alpha)}\cdot \|f\|_{H^p}\]
proving that $J_ g:H^p\rightarrow H^q$ is bounded with $\|J_ g\|_{H^p\rightarrow H^q}\le C \|g\|_{\Lambda(\alpha)}$.\\

Finally, we consider the case  $0<q<2$. Let $t=(2-q)p/q$ and observe that $2-t>p$ since $p<q$. We use the area function description of Hardy spaces, \eqref{form} and H\"{o}lder's inequality to obtain
\begin{displaymath}
\begin{split}
\|J_ g f\|_{H^q}^q & \asymp \|A(J_ g f)\|^q_{L^{q}(\Sn)}
\\
&=  \int_{\Sn} \left (\int_{\Gamma(\zeta)} \!\!|f(z)|^2 \,|Rg(z)|^2\,(1-|z|^2)^{1-n} dv(z) \right )^{q/2} \!\!\!\! d\sigma(\zeta)
\\
& \le \int_{\Sn} \!\!|f^*(\zeta)|^{tq/2}\left (\int_{\Gamma(\zeta)} \!\!|f(z)|^{2-t} \,|Rg(z)|^2\,(1-|z|^2)^{1-n} dv(z) \right )^{q/2}\!\!\!\! \!d\sigma(\zeta)
\\
&\le  \|f^*\|_{L^p(\Sn)}^{(2-q)p/2} \!\left (\int_{\Sn} \int_{\Gamma(\zeta)} \!\!|f(z)|^{2-t} \,|Rg(z)|^2\,(1-|z|^2)^{1-n} dv(z) \,d\sigma(\zeta) \!\right )^{q/2} \!\!.
\end{split}
\end{displaymath}
Now, the $L^p$-boundedness of the admissible maximal function (Theorem \ref{NTMT}) gives $\|f^*\|_{L^p(\Sn)}\le C \|f\|_{H^p}$. Also, by \eqref{EqG} and the embedding of Hardy spaces into Bergman spaces (Theorem \ref{Hardy-Bergman}) we have
\begin{displaymath}
\begin{split}
\int_{\Sn} \int_{\Gamma(\zeta)} \!\!|f(z)|^{2-t} &\,|Rg(z)|^2\,(1-|z|^2)^{1-n}dv(z) d\sigma(\zeta)
\\
&\asymp \int_{\Bn} \!\!|f(z)|^{2-t} \,|Rg(z)|^2\,(1-|z|^2)\,dv(z)
\\
&\le C \|g\|^2_{\Lambda(\alpha)}\int_{\Bn} \!\!|f(z)|^{2-t} (1-|z|^2)^{2\alpha-1} dv(z)
\\
& \le C \|g\|^2_{\Lambda(\alpha)}\cdot \|f\|_{H^p}^{2-t}.
\end{split}
\end{displaymath}
All together yields
\[ \|J_ g f\|_{H^q}^q \le C \|g\|^q_{\Lambda(\alpha)}\cdot \|f\|_{H^p}^{(2-q)p/2+(2-t)q/2}=C \|g\|^q_{\Lambda(\alpha)}\cdot \|f\|^q_{H^p}\]
proving that $J_ g:H^p\rightarrow H^q$ is bounded with $\|J_ g\|_{H^p\rightarrow H^q}\le C \|g\|_{\Lambda(\alpha)}$ finishing the proof of the Theorem.

\subsection{Duren's theorem}
The proof of Theorem \ref{mt2} is closely related with Duren's theorem \cite{Du} describing the boundedness of the embedding $I_ d:H^p\rightarrow L^q(\mu)$ for $p<q$ (just look that several terms of the form $\|f\|_{L^q(\mu_ g)}$ appeared in the proof), and the original proof in one dimension used Duren's theorem. Surprisingly, the use of the embedding of Hardy spaces into Bergman spaces makes the proof of Duren's theorem almost trivial. For $s>0$ a finite positive Borel measure on $\Bn$ is called an $s$-Carleson measure if there exists a constant $C>0$ such that $\mu(B_{\delta}(\zeta))\le C \delta \,^{ns}$ for all $\zeta\in \Sn$ and $\delta>0$. It is well known (see \cite[Theorem 45]{ZZ}) that $\mu$ is an $s$-Carleson measure if and only if
\begin{equation}\label{sCM}
\sup_{a\in \Bn}\int_{\Bn} \!\!\left (\frac{1-|a|^2}{|1-\langle a,z \rangle |^2} \right )^{ns} \!\!d\mu(z)<\infty.
\end{equation}
\begin{otherth}[Duren]\label{Durth}
Let $\mu$ be a finite positive Borel measure on $\Bn$ and $0<p<q<\infty$. Then $I_ d:H^p\rightarrow L^q(\Bn,d\mu)$ is bounded if and only if $\mu$ is a $q/p$-Carleson measure.
\end{otherth}
\begin{proof}
By testing the inequality $\int |f|^q d\mu \le C \|f\|^q_{H^p}$ on the functions $f_ a(z)=(1-|a|^2)^{n/p}/(1-\langle z,a\rangle)^{2n/p}$ one gets \eqref{sCM} with $s=q/p$. Conversely, assume that $\mu$ is a $q/p$-Carleson measure. The well known inequality
\begin{displaymath}
|f(z)|^q \lesssim \int_{\Bn} \!\frac{|f(w)|^q}{|1-\langle w,z\rangle |^{n+1+\gamma}} dv_{\gamma}(w)
\end{displaymath}
with $\gamma=2nq/p-n-1>-1$ together with Fubini's theorem, condition \eqref{sCM} and the embedding of Hardy spaces into Bergman spaces gives
\begin{displaymath}
\begin{split}
\int_{\Bn} |f(z)|^q \,d\mu(z) &\le C\int_{\Bn} \!\!|f(w)|^q \left (\int_{\Bn} \!\frac{d\mu(z)}{|1-\langle w,z\rangle|^{2nq/p}} \right ) dv_{\gamma}(w)
\\
&\le C \int_{\Bn} \!\!|f(w)|^q \,(1-|w|^2)^{nq/p-n-1}dv(z)\le C \|f\|^q_{H^p}.
\end{split}
\end{displaymath}
Theorem \ref{Durth} is now proven.
\end{proof}
\section{Proof of Theorem \ref{mt3}}\label{s5}
\subsection{Sufficiency}
This is the easy case. Suppose that $g\in H^r$. The area description of functions in the Hardy space, H\"{o}lder's inequality with exponent $p/q>1$ and the $L^p$-boundedness of the admissible maximal function gives
\begin{displaymath}
\begin{split}
\|J_ g f\|^q_{H^q} &\asymp \int_{\Sn} \left (\int_{\Gamma(\zeta)} \!\!|f(z)|^2 \,|R g(z)|^2\,(1-|z|^2)^{1-n}dv(z)\right )^{q/2} \!\!\!d\sigma(\zeta)
\\
&\le \int_{\Sn} \!(f^*(\zeta) )^{q}\left (\int_{\Gamma(\zeta)} \!\! |Rg(z)|^2\,(1-|z|^2)^{1-n}\,dv(z)\right )^{q/2}\!\!\! d\sigma(\zeta)
\\
&\le \|f^*\|_{L^p(\Sn)}^q \cdot \|A(g)\|_{L^r(\Sn)}^q
\le C \|f\|_{H^p}^{q} \cdot\|g\|_{H^r}^{q},
\end{split}
\end{displaymath}
proving that $J_ g:H^p\rightarrow H^q$ is bounded with $\|J_ g\|_{H^p\rightarrow H^q}\le C \|g\|_{H^r}$.

\subsection{Necessity: first considerations}
\par The proof of the converse implication $J_ g:H^p\rightarrow H^q$ bounded implies $g\in H^r$ with $r=pq/(p-q)$ is much more difficult. Here we will deal with some easy cases as well as some remarks.
First of all, the case $q=2$ is particularly simple. Indeed, by the Littlewood-Paley inequalities, \eqref{form} and Theorem \ref{LT} we have
\begin{displaymath}
\|J_ g f\|^2_{H^ 2}\asymp \int_{\Bn} |f(z)|^2\,|Rg(z)|^2\,(1-|z|^2)\, dv(z)\le C \|f\|^2_{H^p}
\end{displaymath}
if and only if, the admissible area function $Ag$ belongs to $L^{2p/(p-2)}(\Sn)$. Moreover, one has $\|J_ g\|_{H^p\rightarrow H^2}\asymp \|Ag \|_{L^{2p/(p-2)}(\Sn)}$. Since  $r=2p/(p-2)$, an application of Theorem \ref{AreaT} gives
\begin{displaymath}
\|J_ g\|_{H^p\rightarrow H^2}\asymp \|Ag \|_{L^{2p/(p-2)}(\Sn)}\asymp \|g\|_{H^r}.
\end{displaymath}

A remark we must make here is that, as done in the proof of Theorem \ref{mt1}, it is enough to prove the inequality $\|g\|_{H^r}\le C \|J_ g\|_{H^p\rightarrow H^q}$ assuming that $g$ is already in the Hardy space $H^r$.

Taking this into account, the case $r=mp$ for some positive integer $m$ can be done as follows: $g\in H^r$ if and only if $g^m\in H^p$,
and since $g^{m+1}=(m+1)J_ g(g^m)$, then with the notation $f_ m=g^{m}$, the Hardy-Stein inequalities together with the identity \eqref{form} gives
\begin{displaymath}
\begin{split}
\|g\|^r_{H^r}&\asymp \int_{\Bn} |g(z)|^{r-2} \, |R g(z)|^2 \, (1-|z|^2) \,  dv(z)
\\
&=\int_{\Bn} |g(z)|^{mp-2-2m}\, |f_ m(z)|^{2}\,|R g(z)|^2 \,(1-|z|^2)\, dv(z)
\\
&=C\int_{\Bn} |J_g f_ m(z)|^{\frac{mp-2-2m}{m+1}} |R(J_ g f_ m)(z)|^2\, (1-|z|^2) \,dv(z).
\end{split}
\end{displaymath}
Since
 $$\frac{mp-2-2m}{m+1}=\frac{mp}{m+1}-2=q-2,$$
another use of the Hardy-Stein inequalities yields
\begin{displaymath}
\|g\|_{H^r}^r\asymp \|J_ g f_ m\|^q_{H^q} \le \|J_ g\|^q_{H^p\rightarrow H^q}\cdot \|f_ m\|^q_{H^p}=\|J_ g\|^q_{H^p\rightarrow H^q}\cdot\|g\|^{rq/p}_{H^r}.
\end{displaymath}
Since $r-rq/p=q$, this clearly implies the desired inequality
\begin{displaymath}
\|g\|_{H^r}\le C \|J_ g\|_{H^p\rightarrow H^q}.
\end{displaymath}
\mbox{}
\\
The general case can be done in a similar manner if one is able to prove the following: let $0<q<p<\infty$ and assume that $J_ g:H^p\rightarrow H^q$ is bounded. Then for all $0<q_ 1<q$ and $0<p_ 1<p$ with
$$\frac{1}{q_ 1}-\frac{1}{p_ 1}=\frac{1}{q}-\frac{1}{p}=\frac{1}{r}$$
the operator $J_ g:H^{p_ 1}\rightarrow H^{q_ 1}$ is also bounded with $\|J_ g\|_{H^{p_ 1}\rightarrow H^{q_ 1}}\le C \|J_ g\|_{H^p\rightarrow H^q}$.
Assuming the previous assertion being true, then one takes a positive integer $m$ with $p_ 1:=r/m<p$. Then, by the case considered before, one gets
\[\|g\|_{H^r} \le C \|J_ g\|_{H^{p_ 1}\rightarrow H^{q_ 1}}\le C \|J_ g\|_{H^p\rightarrow H^q}.\]
The proof of the previous claim in the one dimensional setting $n=1$ follows from the factorization of function in Hardy spaces. Indeed, given $f\in H^{p_ 1}(\mathbb{B}_ 1)$ factorize it as $f=f_ 1\cdot f_ 2$ with $f_ 1\in H^p(\mathbb{B}_ 1)$ and $f_ 2\in H^{t}(\mathbb{B}_ 1)$ such that $\|f_ 1\|_{H^p}\cdot \|f_ 2\|_{H^t}\le \|f\|_{H^{p_1}}$. Here $t$ is defined by the relation $1/p_ 1=1/p+1/t$. Then, by the area description of functions in the Hardy spaces, H\"{o}lder's inequality, and the boundedness of the admissible maximal function,
\begin{displaymath}
\begin{split}
\|J_ g f\|^{q_ 1}_{H^{q_ 1}} & \asymp \int_{\mathbb{S}_ 1} \left (\int_{\Gamma(\zeta)} \! |f_ 1(z)|^2\,|f_ 2(z)|^2\,|g'(z)|^2 \,dv(z)\right )^{q_ 1/2} \!\!\! \! d\sigma(\zeta)
\\
&\le \int_{\mathbb{S}_ 1} |f^*_ 2(\zeta)|^{q_ 1}\left (\int_{\Gamma(\zeta)} \!|(J_ g f_ 1)'(z)|^2 \,dv(z)\right )^{q_ 1/2} \!\!\! \! d\sigma(\zeta)
\\
&\le \|f_ 2^*\|^{q_ 1}_{L^{t}(\mathbb{S}_ 1)} \cdot \|J_ g f_ 1\|_{H^q}^{q_ 1}
\\
&\le \|J_ g\|^{q_ 1}_{H^p\rightarrow H^q}\cdot \|f_ 1\|_{H^p}^{q_ 1}\cdot \|f_ 2\|_{H^t}^{q_ 1} \le \|J_ g\|^{q_ 1}_{H^p\rightarrow H^q}\cdot \|f\|_{H^{p_ 1}}^{q_ 1}.
\end{split}
\end{displaymath}
When $n>1$ the factorization theorem is not at our disposal \cite{Go}, and even that there are some weak factorization results available for Hardy spaces $H^p(\Bn)$ for $0<p\le1$ (see \cite{CRW,GL}), we couldn't make effective use of them. Being unable to prove the assertion, at least directly, the proof of the necessity in Theorem \ref{mt3} will follow a different route. We mention here that, once Theorem \ref{mt3} is completely proved, then the previous claim is just a simple consequence of the theorem itself.
\subsection{Necessity: the case $\mathbf{r>2}$}
We recall that the measure $\mu_ g$ is defined as
$d\mu_ g(z)=|Rg(z)|^2 (1-|z|^2)\,dv(z).$
We need first the following simple observation.
\begin{lemma}\label{kL1}
Let $0<s<p<\infty$ and $g\in H(\Bn)$. Then
\begin{displaymath}
\int_{\Bn} |f(z)|^s d\mu_ g(z)\le C \|f\|_{H^p}^s
\end{displaymath}
if and only if $g\in H^{\frac{2p}{p-s}}$. Moreover,
$
\|I_ d\|_{H^p\rightarrow L^s(\mu _ g)}\asymp \|g\|^{2/s}_{H^{\frac{2p}{p-s}}}.
$
\end{lemma}
\begin{proof}
This is an immediate consequence of Theorem \ref{LT} and Theorem \ref{AreaT}.
\end{proof}
Observe that, for $0<s<p$, the number $2p/(p-s)$ is always strictly greater than $2$, so that, for the proof of the necessity in Theorem \ref{mt3} we are only able to apply the previous Lemma in the case $r>2$. So, assume that $J_ g:H^p\rightarrow H^q$ is bounded and $r>2$. By Lemma \ref{kL1}, we have
\begin{equation}\label{HrE1}
\|g\|_{H^r}^2\asymp \sup_{\|f\|_{H^p}=1}\int_{\Bn} |f(z)|^s d\mu_ g(z)
\end{equation}
with $s=p-2(p-q)/q$.
We start first with the case $q> 2$. In that case, $s>2$  and then, by \eqref{EqG} and H\"{o}lder's inequality
\begin{displaymath}
\begin{split}
\int_{\Bn} \!\!|f(z)|^s & d\mu_ g(z)\asymp \int_{\Sn} \int_{\Gamma(\zeta)} \!\!|f(z)|^s \,|R g(z)|^2 \,(1-|z|^2)^{1-n} dv(z)\,d\sigma(\zeta)
\\
& \le \int_{\Sn} \!\!|f^*(\zeta)|^{s-2} \!\left (\int_{\Gamma(\zeta)} \!\!\!|f(z)|^2 \,|R g(z)|^2 \,(1-|z|^2)^{1-n} dv(z) \!\right ) d\sigma(\zeta)
\\
& \le \|f^*\|_{L^p(\Sn)}^{s-2}\cdot \|A(J_ g f)\|^2_{L^q(\Sn)}.
\end{split}
\end{displaymath}
Therefore, using the $L^p$-boundedness of the admissible maximal function together with Theorem \ref{AreaT} we have
\begin{displaymath}
\begin{split}
\int_{\Bn} |f(z)|^s d\mu_ g(z)
&\le C \|f\|^{s-2}_{H^p} \cdot \|J_ g f \|_{H^q}^{2} \le C \|J_ g\|_{H^{p}\rightarrow H^q}^2 \cdot \|f\|^s_{H^p}.
\end{split}
\end{displaymath}
This together with  \eqref{HrE1} gives $\|g\|_{H^r}\le C \|J_ g\|_{H^{p}\rightarrow H^q}$ finishing the proof of this case.\\

Now assume that $q<2$ and $r>2$. Then $0<s<2$. By H\"{o}lder's inequality, the Hardy-Stein inequalities and Lemma \ref{kL1},
\begin{displaymath}
\begin{split}
\|f\|^s_{L^s(\mu_ g)}&\le \left (\int_{\Bn} \!\!|J_ g f(z)|^{\frac{s(2-q)}{2-s}}\,d\mu_{g}(z) \right)^{\frac{2-s}{2}} \!\left (\int_{\Bn}\!\! |J_ g f(z)|^{q-2} |f(z)|^2 d\mu_{g}(z)\right)^{\frac{s}{2}}
\\
& \asymp \left (\int_{\Bn} \!\!|J_ g f(z)|^{\frac{qs}{p}}\,d\mu_{g}(z) \right)^{\frac{2-s}{2}}\,\|J_ g f\|_{H^q}^{\frac{qs}{2}}
\\
&\lesssim \left (\|g\|_{H^r}^{2}\cdot \|J_ g f\|_{H^q}^{qs/p}\right )^{\frac{2-s}{2}} \|J_ g \|_{H^p\rightarrow H^q}^{\frac{qs}{2}}\cdot  \|f\|_{H^p}^{\frac{qs}{2}}
\\
&\le \|g\|_{H^r}^{2-s}\cdot \|J_ g \|_{H^p\rightarrow H^q}^s \cdot \|f\|_{H^p}^s.
\end{split}
\end{displaymath}
Therefore, using \eqref{HrE1} we get
\[\|g\|_{H^r}^2 \le C \|g\|_{H^r}^{2-s}\cdot \|J_ g \|_{H^p\rightarrow H^q}^s,\]
and this implies that
$\|g\|_{H^r}\le C \|J_ g\|_{H^p\rightarrow H^q}$
as desired. This finishes the proof for $r>2$.

\subsection{Necessity: the case $\mathbf{r\le 2}$}
In order to obtain the remainder case, we must extend Lemma \ref{kL1} in order to obtain a description of $H^r$ functions in terms of Carleson type embeddings with $r\le 2$. This is what we are doing next.
\begin{lemma}\label{KL2}
Let $g\in H(\Bn)$, $0<s<p<\infty$ and $0<t<1$. Then
\begin{displaymath}
\int_{\Bn} |f(z)|^s \, |g(z)|^{2t-2}\,d\mu_ g(z)\le C \|f\|_{H^p}^s
\end{displaymath}
if and only if $g\in H^{\frac{2pt}{p-s}}$. Moreover, if $\widehat{\mu}_ g$ is the measure defined by
$d\widehat{\mu}_ g(z)=|g(z)|^{2t-2}\,d\mu_ g(z),$
then
$$\|I_ d\|_{H^p\rightarrow L^s(\widehat{\mu} _ g)}\asymp \|g\|^{2t/s}_{H^{\frac{2pt}{p-s}}}.$$
\end{lemma}
\begin{proof}
The proof is a direct consequence of Theorem \ref{LT} and Theorem \ref{KP}  below, that generalizes the description of Hardy spaces in terms of the area function.
\end{proof}
\begin{theorem}\label{KP}
Let $g\in H(\Bn)$ and $0<p,t<\infty$. Then $g\in H^{pt}$ if and only if
$$ I_ {p,t}(g):=\int_{\Sn}\left (\int_{\Gamma(\zeta)} \!\!|g(z)|^{2t-2}\,|Rg(z)|^2 (1-|z|^2)^{1-n} dv(z) \right )^{p/2} \!\!\!\!d\sigma(\zeta)<\infty.$$
Moreover, if $g(0)=0$, we have
$\|g\|_{H^{pt}} \asymp I_{p,t}(g)^{1/pt}.$
\end{theorem}

Before going to the proof of Theorem \ref{KP}, now we use Lemma \ref{KL2} to obtain the necessity in Theorem \ref{mt3} for $r\le 2$. Since always one has $q<r$ it is possible to choose $0<t<1$ with $q<2t<r$. Let $s=p-2t\frac{(p-q)}{q}$. Then $0<s<p$ and also $0<s<2$. By Lemma \ref{KL2},
\begin{equation}\label{ET0}
\|g\|_{H^r}^{2t}\asymp \sup_{\|f\|_{H^p}=1}\int_{\Bn} \!|f(z)|^s |g(z)|^{2t-2}\,d\mu_ g(z).
\end{equation}
For $f\in H^p$, by H\"{o}lder's inequality,  we have
\begin{equation}\label{EqT1}
\begin{split}
\int_{\Bn} \!\!\!|f(z)|^s |g(z)|^{2t-2} d\mu_ g(z)\le& \left (\int_{\Bn} \!\!\!|J_ g f(z)|^{\frac{s(2-q)}{2-s}}\,|g(z)|^{(2t-2)\cdot \frac{2}{2-s}}\,d\mu_{g}(z) \! \right)^{\frac{2-s}{2}}
\\
&\times \left (\int_{\Bn} \!|J_ g f(z)|^{q-2} |f(z)|^2 d\mu_{g}(z)\right)^{s/2}.
\end{split}
\end{equation}
Observe that $\frac{s(2-q)}{2-s}<q$ if and only if $s<q$ and this holds if $q<2t$. Let
$$s_ q=\frac{s(2-q)}{2-s}; \qquad t_ s=\frac{2t-s}{2-s}.$$
We have $0<s_ q<q$ and $0<t_ s<1$. Then, by Lemma \ref{KL2}
\begin{displaymath}
\begin{split}
\int_{\Bn}\!\!|J_ g f(z)|^{\frac{s(2-q)}{2-s}}\,|g(z)|^{(2t-2)\cdot \frac{2}{2-s}}\,d\mu_{g}(z)& =\int_{\Bn}\!\!|J_ g f(z)|^{s_ q}\,|g(z)|^{2t_ s-2}\,d\mu_{g}(z)
\\
&\le C \|g\|^{2t_ s}_{H^{\gamma}}\cdot \|J_ g f\|_{H^q}^{s_ q},
\end{split}
\end{displaymath}
with
\begin{displaymath}
\gamma=\frac{2q \cdot t_ s}{q-s_ q}=\frac{pq}{p-q}=r.
\end{displaymath}
Putting this into \eqref{EqT1} and using the Hardy-Stein inequalities, we obtain
\begin{displaymath}
\begin{split}
\int_{\Bn} \!\!|f(z)|^s |g(z)|^{2t-2}\,d\mu_ g(z) &\lesssim \big ( \|g\|^{2t_ s}_{H^{r}}\cdot \|J_ g f\|_{H^q}^{s_ q} \big )^{1-s/2}\cdot \|J_ g f\|_{H^q}^{qs/2}
\\
& =\|g\|^{2t-s}_{H^r} \cdot \|J_ g f \|_{H^q}^{s}
\\
& \le \|g\|^{2t-s}_{H^r} \cdot \|J_ g\|_{H^p\rightarrow H^q}^s \cdot\| f \|_{H^p}^{s}.
\end{split}
\end{displaymath}
Taking the supremum over all $f$ in $H^p$ with $\|f\|_{H^p}=1$ and using \eqref{ET0} we get
\begin{displaymath}
\|g\|_{H^r}^{2t} \lesssim \|g\|^{2t-s}_{H^r} \cdot \|J_ g\|_{H^p\rightarrow H^q}^s
\end{displaymath}
that clearly implies the inequality
$ \|g\|_{H^r} \le C \|J_ g\|_{H^p\rightarrow H^q}$
finishing the proof of the Theorem.

\subsection{Proof of Theorem \ref{KP}}
The case $t=1$ is just Theorem \ref{AreaT} but our proof also includes this case.
The case $p=2$ is obvious due to \eqref{EqG} and the Hardy-Stein inequalities. To deal with the other cases, as done before, using standard approximation arguments it is enough to establish the corresponding inequalities assuming that both $\|g\|_{H^{pt}}$ and $I_{p,t}(g)$ are finite.
\subsubsection{Step $1$} For $p>2$ we prove that
\begin{equation}\label{ES1}
\|g\|_{H^{pt}}^{pt} \le C \,I_{p,t}(g).
\end{equation}
By the Hardy-Stein inequalities, \eqref{EqG}, H\"{o}lder's inequality and the $L^p$  boundedness of the admissible maximal function, we have
\begin{displaymath}
\begin{split}
\| g\|^{pt}_{H^{pt}}
&\asymp \int_{\Sn} \left (\int_{\Gamma(\zeta)} \!\!|g(z)|^{pt-2}|R g(z)|^2\,(1-|z|^2)^{1-n} dv(z)\right ) d\sigma(\zeta)
\\
&\le  \int_{\Sn} \!\!|g^*(\zeta)|^{pt-2t}\left (\int_{\Gamma(\zeta)}\!\! |g(z)|^{2t-2}|R g(z)|^2\,(1-|z|^2)^{1-n} dv(z)\right )\! d\sigma(\zeta)
\\
& \le \|g^*\|_{L^{pt}(\Sn)}^{t(p-2)}\cdot I_{p,t}(g)^{2/p}
 \le C\,\|g\|^{pt-2t}_{H^{pt}}\cdot I_{p,t}(g)^{2/p},
\end{split}
\end{displaymath}
and this clearly gives the inequality \eqref{ES1}.
\subsubsection{Step $2$}
We show that, for $0<p<2$, one has
\begin{displaymath}
I_{p,t} (g)\le C \,\|g\|_{H^{pt}}^{pt}.
\end{displaymath}
To prove the inequality, apply H\"{o}lder's inequality, Theorem \ref{NTMT}, \eqref{EqG} and the Hardy-Stein inequalities to obtain
\begin{displaymath}
\begin{split}
I_ {p,t}(g)&=\int_{\Sn}\left (\int_{\Gamma(\zeta)} \!\!|g(z)|^{2t-2}\,|Rg(z)|^2 \,(1-|z|^2)^{1-n} dv(z) \right )^{p/2}\!\!\!\! d\sigma(\zeta)
\\
& \le \int_{\Sn} \!\!|g^*(\zeta)|^{\frac{(2-p)tp}{2}}\left (\int_{\Gamma(\zeta)} \! \!|g(z)|^{pt-2}\,|Rg(z)|^2 \,(1-|z|^2)^{1-n} dv(z) \right )^{p/2} \!\!\!\!\!d\sigma(\zeta)
\\
& \le \|g^*\|_{L^{pt}(\Sn)}^{pt(1-p/2)} \left (\int_{\Sn}\int_{\Gamma(\zeta)} \!\!|g(z)|^{pt-2}\,|Rg(z)|^2 \,(1-|z|^2)^{1-n} dv(z) \,d\sigma(\zeta)\right )^{p/2}
\\
&\le C \|g\|^{pt}_{H^{pt}}.
\end{split}
\end{displaymath}
Notice that the same method shows that, if $u_{\varphi}$ is the invariant Poisson integral of a function $\varphi \in L^{pt}(\Sn)$, and $p<2$ with $pt>1$ then one has
\begin{equation}\label{EqH1}
\int_{\Sn} \!\!\left (\!\int_{\Gamma(\zeta)}\! \!\!|u_{\varphi}(z)|^{2t-2} |\widetilde{\nabla} u_{\varphi}(z)|^2 \, d\lambda_ n (z) \!\right )^{p/2} \!\!\!\!\!\!d\sigma(\zeta)\le C \,\|\varphi\|^{pt}_{L^{pt}(\Sn)}.
\end{equation}
Indeed, we also have the Hardy-Stein inequalities for $u_{\varphi}$ and the boundedness of the admissible maximal function $\|u^*_{\varphi}\|_{L^p(\Sn)}\le C \|\varphi\|_{L^p(\Sn)}$ for $1<p<\infty$ (see \cite[Theorem 5.4.10]{Rud}).
\subsubsection{Step $3$} For $p>2$ we establish the inequality
\begin{equation*}
I_{p,t} (g)\le C \,\|g\|_{H^{pt}}^{pt}.
\end{equation*}
We begin with the case $p\ge 4$. The case $2<p<4$ will be deduced later from this case. Since
$ |Rg(z)|\le |\nabla g(z)|\le (1-|z|^2)^{-1}\,|\widetilde {\nabla}g(z)|$ (see \cite[Lemma 2.14]{ZhuBn}),
it is enough to show that
\begin{equation}\label{ES3-I}
J_{p,t} (g)\le C \,\|g\|_{H^{pt}}^{pt},
\end{equation}
where
$$J_ {p,t}(g):= \int_{\Sn} \left (\int_{\Gamma(\zeta)} |g(z)|^{2t-2} \,|\widetilde{\nabla} g (z)|^2\,d\lambda_ n(z)\right )^{p/2}\!\!\!\!d\sigma(\zeta).$$
 We follow an argument in \cite[p. 282]{Stein}, but with the use of the invariant Green's formula instead of the classical one. By duality, we have
\begin{equation}\label{DR}
J_{p,t}(g) ^{2/p} \asymp \sup \int_{\Sn} \!\!\left (\int_{\Gamma(\zeta)}\!\!\!|g(z)|^{2t-2}\,|\widetilde{\nabla} g(z)|^2 \, d\lambda_ n(z) \!\right ) \varphi(\zeta)\,d\sigma(\zeta),
\end{equation}
where the supremum runs over all positive functions $\varphi$ in $L^{p/(p-2)}(\Sn)$ with $\|\varphi\|_{L^{p/(p-2)}(\Sn)}=1$.  Since $1-|z|^2$ is
comparable to $|1-\langle z,\zeta\rangle |$ for $z$ in $\Gamma(\zeta)$, we have
\begin{equation}\label{DR2}
\begin{split}
\int_{\Sn} &\left (\int_{\Gamma(\zeta)} \!\!|g(z)|^{2t-2}\,|\widetilde{\nabla} g(z)|^2\, d\lambda_ n(z) \right )\,\varphi(\zeta)\, d\sigma (\zeta)
\\
& \asymp \int_{\Sn} \left (\int_{\Gamma(\zeta)} \!\!|g(z)|^{2t-2}\,|\widetilde{\nabla} g(z)|^2 \frac{(1-|z|^2)^{2n}}{|1-\langle z,\zeta\rangle |^{2n}}\,d\lambda_ n(z) \right )\,\varphi(\zeta)\, d\sigma(\zeta)
\\
& \le \int_{\Bn} \!\! |g(z)|^{2t-2}\,|\widetilde{\nabla} g(z)|^2 \,u_{\varphi}(z)\,(1-|z|^2)^n\,d\lambda_ n(z).
\end{split}
\end{equation}
where $u_{\varphi}$ denotes the invariant Poisson integral of the function $\varphi$. An elementary calculation shows that
\[ \widetilde{\Delta} (|g|^{2t})(z)=4t^2 |g(z)|^{2t-2}\, |\widetilde{\nabla}_ h g(z)|^2=t^2|g(z)|^{2t-2}\, |\widetilde{\nabla} g(z)|^2,\qquad z\in \Bn\]
 where $\widetilde{\Delta}$ is the invariant Laplace operator. If $t<1$ the last identity holds at the points $z\in \Bn$ with $g(z)\neq 0$. Therefore, the last integral in \eqref{DR2} is equal to
\begin{displaymath}
 \frac{1}{t^2} \int_{\Bn} \widetilde{\Delta} (|g|^{2t})(z) \,u_{\varphi}(z)\,(1-|z|^2)^n\,d\lambda_ n(z).
\end{displaymath}
 Using that $u_{\varphi}$ is $\mathcal{M}$-harmonic on $\Bn$ and the identity $\widetilde{\Delta} (U\cdot V)=U \widetilde{\Delta} V+V \widetilde{\Delta} U +2 \langle \widetilde{\nabla} U, \widetilde{\nabla} V \rangle _{\mathbb{R}}$, where $\langle \cdot, \cdot \rangle_{\mathbb{R}}$ denotes the inner product in $\mathbb{R}^{2n}$,  we see that the previous integral is dominated by
$$I_ 1(g,\varphi)+I_ 2(g,\varphi)$$
with
\begin{displaymath}
I_ 1(g,\varphi)=\int_{\Bn} \widetilde{\Delta} (u_{\varphi} \, |g|^{2t})(z)\,(1-|z|^2)^n\,d\lambda_ n(z)
\end{displaymath}
and
\begin{displaymath}
I_ 2(g,\varphi)=\int_{\Bn} |\widetilde{\nabla} (|g|^{2t})(z)|\cdot |\widetilde{\nabla} u_{\varphi}(z)|\,(1-|z|^2)^n\,d\lambda_ n(z).
\end{displaymath}
Since $(1-|z|^2)^n \lesssim  G(z)$, where $G$ is the invariant Green's function, the term $I_ 1(g,\varphi)$ is estimated using the invariant Green's formula and H\"{o}lder's inequality to obtain
\begin{equation}\label{Term1}
I_ 1(g,\varphi) \le C \int_{\Sn} |g(\zeta)|^{2t}\,\varphi(\zeta)\,d\sigma(\zeta)
\le C\,\|g\|_{H^{pt}}^{2t} \cdot \|\varphi\|_{L^{p/(p-2)}(\Sn)}.
\end{equation}
Notice that there is no problem with the use of the invariant Green's formula if $t\ge 1$ because in that case, the function $|g|^{2t}$ is of class $C^2$. When $0<t<1$ one uses standard approximation arguments, for example replacing $|g|^{2t}$ by $(|g|^2+\varepsilon)^t$ and then letting $\varepsilon \rightarrow 0$.

In order to estimate the second term $I_ 2(g,\varphi)$, first we use that $\big |\widetilde{\nabla}(|g|^{2t})(z)\big |\asymp |g(z)|^{2t-1} |\widetilde{\nabla} g(z)|$ to get
\begin{displaymath}
I_ 2(g,\varphi)\asymp \int_{\Bn} |g(z)|^{2t-1}\,|\widetilde{\nabla} g(z)|\cdot |\widetilde{\nabla} u_{\varphi}(z)|\,(1-|z|^2)^n\,d\lambda_ n (z).
\end{displaymath}
If $p=4$, an application of Cauchy-Schwarz together with the Hardy-Stein inequalities yield
\begin{displaymath}
\begin{split}
I_ 2(g,\varphi) &\lesssim \left (\int_{\Bn} \!\!\!|g(z)|^{4t-2}\,|\widetilde{\nabla} g(z)|^2 \,(1-|z|^2)^n\,d\lambda_ n(z) \!\right )^{\frac{1}{2}}
 \!\left (\int_{\Bn} \!\!\!|\widetilde{\nabla} u_{\varphi} (z)|^2 \,(1-|z|^2)^n\,d\lambda_ n(z) \!\right )^{\frac{1}{2}}
 \\
& \lesssim \|g\|_{H^{4t}}^{2t}\cdot \|\varphi\|_{L^2(\Sn)}.
 \end{split}
\end{displaymath}
Bearing in mind \eqref{DR}, \eqref{DR2} and \eqref{Term1}, this gives $J_{4,t}(g)^{1/2}\le C \|g\|_{H^{4t}}^{2t}$ proving the desired result when $p=4$.\\

If $p>4$ then $1<\frac{p}{p-2}<2$ and it has been already proved in \eqref{EqH1} that
\begin{equation}\label{EqS3-0}
\int_{\Sn} \!\!\left (\!\int_{\Gamma(\zeta)}\! \! |\widetilde{\nabla} u_{\varphi}(z)|^2 \, d\lambda_ n(z) \!\right )^{\!\!\frac{p}{2(p-2)}} \!\!\!\!\!\!\!\!d\sigma(\zeta)\le C \,\|\varphi\|^{p/(p-2)}_{L^{p/(p-2)}(\Sn)}.
\end{equation}
 By \eqref{EqG} and H\"{o}lder's inequality we have
\begin{displaymath}
\begin{split}
I_ 2(g,\varphi) &\lesssim \int_{\Sn} \left (\int_{\Gamma(\zeta)}\!\! |g(z)|^{2t-1} |\widetilde{\nabla} g(z)|\,|\widetilde{\nabla} u_{\varphi}(z)| \, d\lambda_ n(z)\right ) d\sigma(\zeta)
\\
&\le \int_{\Sn} \!\!\!|g^*(\zeta)|^t\left (\int_{\Gamma(\zeta)}\!\! |g(z)|^{t-1} |\widetilde{\nabla} g(z)|\,|\widetilde{\nabla} u_{\varphi}(z)|\, d\lambda_ n(z)\right ) d\sigma(\zeta)
\\
&\le \|g^*\|^t_{L^{pt}(\Sn)}\cdot I_ 3(g,\varphi)^{(p-1)/p},
\end{split}
\end{displaymath}
with
\begin{displaymath}
I_ 3(g,\varphi)=\int_{\Sn} \left (\int_{\Gamma(\zeta)}\!\! |g(z)|^{t-1} |\widetilde{\nabla} g(z)|\,|\widetilde{\nabla} u_{\varphi}(z)| \, d\lambda_ n(z)\right )^{p/(p-1)} \!\!\!\!\!\!\!\!d\sigma(\zeta).
\end{displaymath}
An application of Theorem \ref{NTMT} gives
\begin{equation}\label{ES3-1}
I_ 2(g,\varphi) \le C \|g\|^t_{H^{pt}}\cdot I_ 3(g,\varphi)^{(p-1)/p}.
\end{equation}
Now, applying Cauchy-Schwarz inequality together with another use of H\"{o}lder's inequality (now with exponent $p-1>1$) and the inequality \eqref{EqS3-0} it follows that
\begin{displaymath}
I_ 3(g,\varphi) \le J_{p,t}(g)^{1/(p-1)}\cdot \|\varphi\|_{L^{p/(p-2)}(\Sn)}^{p/(p-1)}.
\end{displaymath}
Putting this inequality into \eqref{ES3-1} we get
\[ I_ 2(g,\varphi) \le C \|g\|^t_{H^{pt}}\cdot J_{p,t}(g) ^{1/p}\cdot \|\varphi\|_{L^{p/(p-2)}(\Sn)}. \]
Taking into account \eqref{DR}, \eqref{DR2} and \eqref{Term1}, this gives
\begin{displaymath}
J_{p,t}(g)^{2/p}\lesssim \|g\|^{2t}_{H^{pt}}+\|g\|^{t}_{H^{pt}}\cdot J_{p,t}(g)^{1/p},
\end{displaymath}
but, since $p>2$, we have already proved in Step $1$ that $$\|g\|_{H^{pt}}^t \lesssim I_{p,t}(g)^{1/p}\le J_{p,t}(g)^{1/p}.$$
Therefore we finally obtain
\[ J_{p,t}(g)^{2/p}\le C \|g\|^{t}_{H^{pt}}\cdot J_{p,t}(g)^{1/p} ,\]
and this clearly implies the inequality \eqref{ES3-I} finishing the proof of that case.\\

It remains to deal with the case $2<p<4$. Since $2p>4$, the previous case gives
\[ I_{2p,t/2}(g) \le C \|g\|_{H^{pt}}^{pt}.\]
Then, by Cauchy-Schwarz inequality and Theorem \ref{NTMT}, we have
\begin{displaymath}
\begin{split}
I_{p,t}(g) &\le \int_{\Sn}\!\!|g^*(\zeta)|^{tp/2}  \left (\int_{\Gamma(\zeta)}\!\!|g(z)|^{t-2}\,|Rg(z)|^{2} \,(1-|z|^2)^{1-n} dv(z)
 \right )^{p/2} \!\!\!\!d\sigma(z)
\\
& \le \|g^*\|_{L^{pt}(\Sn)}^{pt/2}\cdot I_{2p,t/2}(g)^{1/2}
 \le C \|g\|_{H^{pt}}^{pt}.
\end{split}
\end{displaymath}
\subsubsection{Step $4$}
Finally, for $0<p<2$, we show that
\begin{equation}\label{ES4}
 \|g\|^{pt}_{H^{pt}}\le C \,I_{p,t}(g).
 \end{equation}
 By the Hardy-Stein inequalities together with \eqref{EqG}, we have
\begin{displaymath}
\|g\|^{pt}_{H^{pt}} \asymp \int_{\Sn} \left (\int_{\Gamma(\zeta)} \!\!|g(z)|^{pt-2} |Rg(z)|^2 \,(1-|z|^2)^{1-n} dv(z)\right )d\sigma(\zeta).
\end{displaymath}
 Then apply H\"{o}lder's inequality with exponent $4/p$ and Cauchy-Schwarz  to get
\begin{displaymath}
\begin{split}
\|g\|^{pt}_{H^{pt}}
&\lesssim \int_{\Sn} \!\!\!\left (\int_{\Gamma(\zeta)} \!\!|g(z)|^{2t-2} \, \frac{|R g(z)|^2 \,dv(z)}{(1-|z|^2)^{n-1}} \!\right )^{\frac{p}{4}} \!\left (\int_{\Gamma(\zeta)} \!\!|g(z)|^{\frac{2tp}{4-p}-2}  \, \frac{|R g(z)|^2 \,dv(z)}{(1-|z|^2)^{n-1}} \!\right )^{\frac{4-p}{4}} \!\!\!\!\!d\sigma(\zeta)
\\
&\le I_{p,t}(g)^{1/2}\cdot I_{4-p,\frac{tp}{4-p}}(g)^{1/2}.
\end{split}
\end{displaymath}
By the case already proved (Step $3$) we have
$$I_{4-p,\frac{tp}{4-p}}(g) \le C \|g\|_{H^{pt}}^{pt},$$
and this clearly establishes \eqref{ES4} finishing the proof of the Theorem.

\subsubsection{Remarks} The same argument shows that, for a function $g\in H(\Bn)$ and $0<p<\infty$, one has $g\in H^{pt}$ if and only if
\begin{displaymath}
\int_{\Sn} \left (\int_{\Gamma(\zeta)} \!\!|g(z)|^{2t-2} |\widetilde{\nabla} g(z)|^2 \, d\lambda_ n(z) \!\right )^{p/2} \!\!\!\!\!\!d\sigma(\zeta)<\infty.
\end{displaymath}
Also, the same method shows that, if $u_{\varphi}$ denotes the invariant Poisson integral of $\varphi$ and $p,t$ are positive numbers with $1<pt<\infty$, then $\varphi$ belongs to $L^{pt}(\Sn)$ if and only if
\begin{displaymath}
\int_{\Sn} \left (\int_{\Gamma(\zeta)} \!\!|u_{\varphi}(z)|^{2t-2} |\widetilde{\nabla} u_{\varphi}(z)|^2 \, d\lambda_ n(z) \!\right )^{p/2} \!\!\!\!\!\!d\sigma(\zeta)<\infty.
\end{displaymath}

\section{Compactness and membership in Schatten classes}\label{s6}
\subsection{Compactness}
It is well known that a linear operator $T:H^p\rightarrow H^q$ is compact if and only if $\|Tf_ k\|_{H^q}\rightarrow 0$ for every bounded sequence $\{f_ k\}$ in $H^p$ converging to zero uniformly on compact subsets of $\Bn$.
With all that has been done in the previous sections, it is now routine the obtention of the corresponding descriptions about the compactness of the integration operator $J_ g:H^p\rightarrow H^q$. We need first the following easy result.
\begin{lemma}\label{LComp1}
Let $0<p,q<\infty$. If $\alpha=n\big (\frac{1}{p}-\frac{1}{q}\big )<1$ then $J_ p:H^p\rightarrow H^q$ is compact for any holomorphic polynomial $p(z)$.
\end{lemma}
\begin{proof}
Let $\{f_ k\}$ be a bounded sequence in $H^p$ converging to zero uniformly on compact subsets of $\Bn$, and fix $0<\varepsilon<1$. Take $0<r<1$ with $1-r<\varepsilon$  and then choose $k_ 0$ such that $\sup_{|z|\le r}|f_ k(z)|<\varepsilon$ for all $k\ge k_ 0$. Then, using Theorem \ref{AreaT}, one easily gets
\begin{displaymath}
\|J_ p f_ k\|_{H^q}^q \asymp \|A(J_ p f_ k)\|_{L^q(\Sn)}^q\lesssim \varepsilon^q \|p\|_{H^q}^q+\|Rp\|_{\infty}^q\cdot I(f_ k)
\end{displaymath}
with
\begin{displaymath}
I(f_ k)=\int_{\Sn} \left (\int_{\Gamma(\zeta)\cap \{|z|>r\}}\! |f_ k(z)|^2\,(1-|z|^2)^{1-n}dv(z)\right )^{q/2}\!\!\! d\sigma(\zeta).
\end{displaymath}
If $\alpha\le 0$, that is, when $q\le p$, by H\"{o}lder's inequality and Theorem \ref{NTMT}, we have
\begin{displaymath}
\begin{split}
I(f_ k)&\le \int_{\Sn} \!\!|f^*_ k(\zeta)|^q \left ( \int_{\Gamma(\zeta)\cap \{|z|>r\}} \!\!\!(1-|z|^2)^{1-n}dv(z) \!\right )^{q/2} \!\!\!\!\!d\sigma(\zeta)
\\
& \asymp (1-r)^q\|f^*_ k\|^q_{L^q(\Sn)} \le C \,\varepsilon^q \|f_ k\|^q_{H^p}\le C\,\varepsilon^q.
\end{split}
\end{displaymath}
If $0<\alpha<1$, then $q>p$ and the standard estimate for $H^p$ functions \cite [Theorem 4.17]{ZhuBn} and Theorem \ref{NTMT} gives
\begin{displaymath}
\begin{split}
I(f_ k)
& \le \|f_ k\|_{H^p}^{q-p} \int_{\Sn} |f^*_ k(\zeta)|^p\left (\int_{\Gamma(\zeta)\cap \{|z|>r\}} \!\! (1-|z|^2)^{-2\alpha+1-n}dv(z)\right )^{q/2}\!\!\!\!
d\sigma(\zeta)
\\
& \le C (1-r)^{q(1-\alpha)} \|f^*_ k\|^p_{L^p(\Sn)}\le C \,\varepsilon^{q(1-\alpha)} \|f_ k\|^p_{H^p} \le C \,\varepsilon^{q(1-\alpha)}.
\end{split}
\end{displaymath}
This proves that $J_ p:H^p\rightarrow H^q$ is compact when $\alpha<1$.
\end{proof}
Now we are ready to state and prove the results on the compactness of $J_ g:H^p\rightarrow H^q$.
Recall that the space of holomorphic functions of vanishing mean oscillation $VMOA$ is the closure of the holomorphic polynomials in $BMOA$.
\begin{theorem}\label{Tc1}
Let $0<p<\infty$ and $g\in H(\Bn)$. Then
$J_ g$ is compact on $H^p$ if and only if $g\in VMOA$.
\end{theorem}
\begin{proof}
If $g$ is in  $VMOA$ then there are holomorphic polynomials $p_ k$ with $\|g-p_ k\|_{BMOA}\rightarrow 0$. By Lemma \ref{LComp1}, the operator $J_{p_ k}$ is compact on $H^p$. From the estimate obtained in Theorem \ref{mt1} it follows that
\[ \|J_ g-J_{p_ k}\|=\|J_{g-p_ k}\|\le C \|g-p_ k\|_{BMOA}\rightarrow 0.\]
Hence $J_ g$ can be approximated by compact operators in the operator norm proving that $J_ g$ is compact.

Conversely, suppose that $J_ g$ is compact on $H^p$. We want to show that $g$ belongs to $VMOA$ or, equivalently, that $\|f_ k\|_{L^p(\mu_ g)}\rightarrow 0$ for any sequence $\{f_ k\}$ of functions in the Hardy space $H^p$ with $\sup \|f_ k\|_{H^p}\le C$ converging to zero uniformly on compact subsets of $\Bn$ \cite[Chapter 5]{ZhuBn}. Since $J_ g$ is compact, we have $\lim \|J_ g f_ k\|_{H^p}=0$. If $p=2$ the result is obvious from \eqref{Eqp2}. For the other values of $p$, notice that in the course of the proof of Theorem \ref{mt1} the following inequalities had been proved
 $$\|f_ k\|^p_{L^p(\mu_ g)} \le C \|f_ k\|_{H^p}^{p-2}\cdot \|J_ g f_ k\|_{H^p}^2\quad \textrm{ if }\quad p>2,$$
 and
 $$\|f_ k\|^p_{L^p(\mu_ g)} \le C \|g\|_{BMOA}^{2-p}\cdot \|J_ g f_ k\|_{H^p}^p\quad \textrm{ if }\quad 0<p<2.$$
 Since $\|f_ k\|_{H^p}\le C$ and $\|J_ g f_ k\|_{H^p}\rightarrow 0$ this shows that $\lim \|f_ k\|_{L^p(\mu_ g)}=0$ proving that $g$ is in $VMOA$.
\end{proof}
Now, for $0<\alpha<1$, we need to introduce the little Lipschitz type space $\lambda(\alpha)$ that consists of those functions $g\in H(\Bn)$ with
$$\lim_{|z|\rightarrow 1^{-}} (1-|z|^2)^{1-\alpha} \,|Rg(z)|=0.$$
\begin{theorem}\label{Tc2}
Let $0<p<q<\infty$, $g\in H(\Bn)$ and $\alpha=n\big (\frac{1}{p}-\frac{1}{q}\big )$. If $\alpha<1$, then the operator
$J_ g:H^p\rightarrow H^q$ is compact if and only if $g\in \lambda(\alpha)$. If $\alpha=1$, then $J_ g:H^p\rightarrow H^q$ is compact if and only if $J_ g\equiv 0$.
\end{theorem}
\begin{proof}
One implication is a consequence of Lemma \ref{LComp1} together with the inequality $\|J_ g\|_{H^p\rightarrow H^q}\le C \|g\|_{\Lambda(\alpha)}$ obtained in Theorem \ref{mt2}, since $\lambda(\alpha)$ is the closure of the holomorphic polynomials in $\Lambda(\alpha)$ \cite[Chapter 7]{ZhuBn}. The other implication follows from the estimate
\begin{displaymath}
|f(z)|\,|Rg(z)| \le C (1-|z|^2)^{-(n+q)/q}\,\|J_ g f\|_{ H^q}
\end{displaymath}
obtained in the proof of Theorem \ref{mt2}. Indeed, if $\{a_ k\}$ is any sequence of points in $\Bn$ with $|a_ k|\rightarrow 1$, consider the functions
\[f_ k(z)=\frac{(1-|a_ k|^2)^{n/p}}{(1-\langle z,a_ k \rangle )^{2n/p}},\qquad z\in \Bn.\]
The functions $f_ k$ are unit vectors on $H^p$ converging to zero uniformly on compact subsets of $\Bn$. Therefore, if $J_ g:H^p\rightarrow H^q$ is compact, then $\|J_ g f_ k\|_{H^q}\rightarrow 0$ and it follows from the previous estimate that
\begin{displaymath}
\begin{split}
(1-|a_ k|^2)^{1-\alpha} |R g(a_ k)|&=(1-|a_ k|^2)^{(n+q)/q} |f_ k(a_ k)|\,|Rg(a_ k)|
\\
&\le C \|J_ g f_ k \|_{H^q} \rightarrow 0,
\end{split}
\end{displaymath}
proving that $g$ belongs to $\lambda(\alpha)$ for $\alpha<1$. If $\alpha=1$ we have proved that $|Rg(z)|\rightarrow 0$ as $|z|\rightarrow 1^{-}$, and hence $g$ must be constant.
\end{proof}
\begin{theorem}\label{Tc3}
Let $0<q<p<\infty$ and $g\in H(\Bn)$. Then
$J_ g:H^p\rightarrow H^q$ is compact if and only if it is bounded, if and only if $g\in H^r$ with $r=pq/(p-q)$.
\end{theorem}
\begin{proof}
Due to Theorem \ref{mt3} it only remains to prove that $J_ g:H^p\rightarrow H^q$ is compact whenever $g$ is in $H^r$. As before, since the holomorphic polynomials are dense in $H^r$, this follows from the inequality $\|J_ g\|_{H^{p}\rightarrow H^q}\le C \|g\|_{H^r}$  in Theorem \ref{mt3} and Lemma \ref{LComp1}.
\end{proof}
\subsection{Schatten classes}
For $0<p<\infty$,  a compact operator $T$ acting on a separable Hilbert space $H$ belongs to the Schatten class $S_ p:=S_ p(H)$ if its sequence of singular numbers belongs to the sequence space $\ell^p$ (the singular numbers are the square roots of the eigenvalues of the positive operator $T^*T$, where $T^*$ is the Hilbert adjoint of $T$). We refer to \cite[Chapter 1]{Zhu} for a brief account on Schatten classes.

Recall that $H^2$ is a reproducing kernel Hilbert space with the reproducing kernel function given by
\[K_ z(w)=\frac{1}{(1-\langle w,z\rangle )^{n}},\quad z,w\in \Bn\]
with norm $\|K_ z\|_{H^2}=\sqrt{K_ z(z)}=(1-|z|^2)^{-n/2}$.   The normalized kernel functions are denoted by $k_ z=K_ z/\|K_ z\|_{H^2}$.  We also need to introduce some ``derivatives" of the kernel functions. For $z,w\in \Bn$ and $t>0$, define
\[K_ z^{t}(w)=\frac{1}{(1-\langle w,z\rangle )^{n+t}}\]
and let $k^{t}_ z$ denote its normalization, that is, $k^{t}_ z=K^{t}_ z/\|K^{t}_ z\|_{H^2}$. Notice that
$K_ z^{t}(w)=\overline{R^{-1,t}K_ w (z)},$
where $R^{-1,t}$ is the unique continuous linear operator on $H(\Bn)$ satisfying
\begin{displaymath}
R^{-1,t}\left (\frac{1}{(1-\langle z,w\rangle )^{n}}\right )=\frac{1}{(1-\langle z,w\rangle )^{n+t}}
\end{displaymath}
for all $w\in \Bn$ (see \cite[Section 1.4]{Zhu}). The operator $R^{-1,t}$ is invertible and its inverse is denoted by $R_{-1,t}$. In particular, since $f
(z)=\langle f,K_ z\rangle_{H^2}$ whenever $f\in H^2$, one has
\begin{equation}\label{eq1}
 R^{-1,t}f (z)=\langle f,K^{t}_ z\rangle_{H^2},\qquad f\in H^2(\Bn).
\end{equation}

In order to describe the membership of the integration operator $J_ g$ in the Schatten ideals $S_ p(H^2)$ we also need the following result that can be of independent interest. A related result in one dimension appears in \cite{Msm}.
\begin{lemma}\label{KLS}
Let $T:H^2(\Bn)\rightarrow H^2(\Bn)$ be a positive operator. For $t> 0$ set
\[\widetilde{T^t}(z)=\langle Tk^{t}_ z,k^{t}_ z\rangle_{H^2},\quad z\in \Bn.\]
\begin{enumerate}
\item[(a)]  Let $0<p\le 1$. If $\widetilde{T^t} \in L^p(\Bn,d\lambda_ n)$ then $T$ is in $S_ p(H^2)$.

\item[(b)] Let $p\ge 1$. If $T$ is in $S_ p(H^2)$ then $\widetilde{T^t} \in L^p(\Bn,d\lambda_ n)$.
\end{enumerate}
\end{lemma}
\begin{proof}
The positive operator $T$ is in $S_ {p}$ if and only if $T^p$ is in the trace class $S_ 1(H^2)$. Fix an orthonormal basis $\{e_ k\}$ of $H^2(\Bn)$. Since $T^p$ is positive, it belongs to the trace class if and only if
\[\sum _ k \langle T^p e_ k,e_ k \rangle_{H^2}  <\infty.\]
Let $S=\sqrt{T^p}$. Then
$\sum _ k \langle T^p e_ k,e_ k \rangle _{H^2} =\sum_ k \|S e_ k\|_{H^2}^2,$
and, by \cite[Theorem 4.41]{ZhuBn}, this is comparable to
\[\sum_ k \|R^{-1,t}S e_ k\|_{A^2_{2t-1}}^2.\]
Now, by \eqref{eq1}, Fubini's theorem and Parseval's identity, we have
\begin{displaymath}
\begin{split}
\sum_ k &\,\|R^{-1,t}S e_ k\|_{A^2_{2t-1}}^2=\sum_ k \int_{\Bn}|R^{-1,t}S e_ k(z) |^2\,dv_{2t-1}(z)
\\
&
=\sum_ k \int_{\Bn} \!\!\big |\langle S e_ k, K_ z^{t} \rangle_{H^2} \big  |^2\,dv_{2t-1}(z)
= \!\int_{\Bn} \!\!\left (\sum_k \big |\langle  e_ k, SK_ z^{t} \rangle_{H^2} \big |^2\right ) dv_{2t-1}(z)
\\
&=\int_{\Bn}\|SK_ z^{t}\|_{H^2}^2\,dv_{2t-1}(z)
=\int_{\Bn}\langle T^p K_ z^{t},K_ z^{t}\rangle_{H^2}\,dv_{2t-1}(z)
\\
&=\int_{\Bn}\langle T^p k_ z^{t},k_ z^{t}\rangle_{H^2} \,\|K_ z^{t}\|_{H^2}^{2}\,dv_{2t-1}(z).
\end{split}
\end{displaymath}
Putting all together and taking into account that $\|K_ z^{t}\|_{H^2}^{2}(1-|z|^2)^{2t-1}$ is comparable to $(1-|z|^2)^{-(n+1)}$, we have that $T$ is in $S_ p$ if and only if
\[\int_{\Bn} \langle T^p k_ z^{t},k_ z^{t}\rangle_{H^2}\,d\lambda_ n(z)<\infty.\]
Now, both (a) and (b) are consequences of the inequalities (see \cite[Proposition 1.31]{Zhu})
\[ \langle T^p k_ z^{t},k_ z^{t}\rangle_{H^2} \le  \left [\langle T k_ z^{t},k_ z^{t}\rangle_{H^2} \right ]^p=[\widetilde{T^t}(z)]^p, \qquad 0<p\le 1\]
and
\[[\widetilde{T^t}(z)]^p=\left [\langle T k_ z^{t},k_ z^{t}\rangle_{H^2} \right ]^p \le \langle T^p k_ z^{t},k_ z^{t}\rangle_{H^2},\qquad p\ge 1.\]
This finishes the proof of the lemma.
\end{proof}
\begin{coro}\label{Cor1}
Let $T:H^2(\Bn)\rightarrow H$ be a bounded linear operator, where $H$
is any separable Hilbert space. Let $t>0$ and consider the function $F_ t(z)=\|Tk_ z^t\|_{H}$. If $p\ge 2$ and $T$ is in $S_ p$ then $F_ t\in L^p(\Bn,d\lambda_ n)$. If $0<p\le 2$ and $F_ t\in L^p(\Bn,d\lambda_ n)$, then $T$ belongs to $S_ p$.
\end{coro}
\begin{proof}
The result is an immediate consequence of Lemma \ref{KLS} since, by definition, $T:H^2\rightarrow H$ is in $S_ p$ if the positive operator $T^* T$ belongs to $S_{p/2}(H^2)$.
\end{proof}

We need the following well known integral estimate that can be found, for example, in \cite[Theorem 1.12]{ZhuBn}.
\begin{otherl}\label{LA}
Let $t>-1$ and $s>0$. There is a positive constant $C$ such that
\[ \int_{\Bn} \frac{(1-|w|^2)^t\,dv(w)}{|1-\langle z,w\rangle |^{n+1+t+s}}\le C\,(1-|z|^2)^{-s}\]
for all $z\in \Bn$.
\end{otherl}
Now we are ready for the description of the membership in $S_ p(H^2)$ of the integration operator $J_ g$.
\begin{theorem}\label{TSCl}
Let $g\in H(\Bn)$. Then
\begin{enumerate}
\item[(a)] For $n<p<\infty$, $J_ g$ belongs to $S_ p(H^2)$ if and only if $g\in B_ p$, that is,
\begin{equation}\label{CBp}
\int_{\Bn} |Rg(z)|^p \,(1-|z|^2)^{p} \,d\lambda_ n(z)<\infty.
\end{equation}
\item[(b)] If $0<p\le n$ then $J_ g$ is in $S_ p(H^2)$ if and only if $g$ is constant.
\end{enumerate}
\end{theorem}
\begin{proof}
Since $B_ p\subset VMOA$, if $g\in B_ p$ then, by Theorem \ref{Tc1}, $J_ g$ is compact and therefore, for $p\ge 2$, it belongs to $S_ p$ if and only if
$\sum_ n \|J_ g e_ k\|_{H^2}^p\le C <\infty$ for all orthonormal sets $\{e_ k\}$ of $H^2$ \cite[Theorem 1.33]{Zhu}. Due to \eqref{Eqp2} we have
\begin{equation}\label{EqSC-1}
\|J_ g e_ k\|_{H^2}^p\asymp \left (\int_{\Bn} \! |e_ k(z)|^2 \,|Rg(z)|^2 \,(1-|z|^2)\,dv(z)\right )^{p/2}.
\end{equation}
Now, since $f\in H^2$ if and only if $R^{-1,1}f \in A^2_ 1(\Bn)\subset A^2_{1+\gamma}(\Bn)$, by using the reproducing formula for the Bergman space $A^2_{1+\gamma}(\Bn)$ to the function $R^{-1,1}e_ k$ one gets
\[ |e_ k(z)|=\big |R_{-1,1}(R^{-1,1}e_ k)(z)\big |\lesssim \int_{\Bn} \frac{|R^{-1,1}e_ k(w)|}{|1-\langle z,w \rangle |^{n+1+\gamma}} \,(1-|w|^2)^{1+\gamma}\, dv(w) \]
with $\gamma>0$ chosen big enough so that all successive applications of Lemma \ref{LA} are going to be correct. Take $0<\varepsilon<1$ with $\varepsilon p<2n$ and apply Cauchy-Schwarz together with Lemma \ref{LA} to obtain
\begin{displaymath}
|e_ k(z)|^2 \lesssim (1-|z|^2)^{-\varepsilon} \int_{\Bn} \frac{|R^{-1,1}e_ k(w)|^2}{|1-\langle z,w \rangle |^{n+1+2\gamma}} \,(1-|w|^2)^{2+2\gamma+\varepsilon}\, dv(w).
\end{displaymath}
Putting this into \eqref{EqSC-1}, using Fubini's theorem, H\"{o}lder's inequality with exponent $p/2$, and taking into account that $\|R^{-1,1}e_ k\|_{A^2_ 1}\lesssim \|e_ k\|_{H^2}$, we obtain
\begin{displaymath}
\begin{split}
\|J_ g e_ k\|_{H^2}^p&\lesssim \left (\int_{\Bn}\!\!|R^{-1,1}e_ k(w)|^2 \,K g(w)\, (1-|w|^2)^{2+2\gamma+\varepsilon} dv(w)\right )^{p/2}
\\
& \lesssim \int_{\Bn}\!\!|R^{-1,1}e_ k(w)|^2 \, Kg(w)^{p/2} \,(1-|w|^2)^{1+\frac{(1+2\gamma+\varepsilon)p}{2}} dv(w)
\end{split}
\end{displaymath}
with
\[ Kg(w):=\int_{\Bn}\!\! \frac{|Rg(z)|^2\,(1-|z|^2)^{1-\varepsilon}\,dv(z)} {|1-\langle z,w \rangle |^{n+1+2\gamma}} .\]
Now, summing on $k$ and using that $$\sum_ k |R^{-1,1}e_ k(w)|^2 \lesssim \|K^1_ w\|_{H^2}^2 \lesssim (1-|w|^2)^{-n-2},$$ we arrive at
\begin{displaymath}
\sum_ k \|J_ g e_ k\|_{H^2}^p \lesssim \int_{\Bn}\!\!Kg(w)^{p/2}\,(1-|w|^2)^{-n-1+\frac{(1+2\gamma+\varepsilon)p}{2}} dv(w).
\end{displaymath}
By H\"{o}lder's inequality and Lemma \ref{LA} we have
\begin{displaymath}
Kg(w)^{p/2} \le \left (\int_{\Bn} \!\!\frac{|Rg(z)|^p\,(1-|z|^2)^{p+t}\,d\lambda_ n(z)} {|1-\langle z,w \rangle |^{n+1+2\gamma}}\right ) (1-|w|^2)^{n+1-t-\frac{(1+\varepsilon)p}{2}-\gamma(p-2)}
\end{displaymath}
with $0<2t<2n-\varepsilon p$. This, together with Fubini's theorem and another application of Lemma \ref{LA} finally gives
\begin{displaymath}
\begin{split}
\sum_ k \|J_ g e_ k\|_{H^2}^p &\lesssim \int_{\Bn} \!\!|Rg(z)|^p\,(1-|z|^2)^{p+t} \left (\int_{\Bn}\!\frac{(1-|w|^2)^{2\gamma-t}dv(w)}{|1-\langle z,w \rangle |^{n+1+2\gamma}}\right ) d\lambda_ n(z)
\\
&\lesssim \int_{\Bn} |Rg(z)|^p\,(1-|z|^2)^{p} \,d\lambda_ n(z)
\end{split}
\end{displaymath}
proving that $J_ g$ belongs to $S_ p(H^2)$. This finishes the proof of the sufficiency in part (a) when $n\ge 2$.\\

Conversely, assume that $J_ g$ belongs to the Schatten class $S_ p(H^2)$ and $p\ge 2$. By Corollary \ref{Cor1}, the function $F(z)=\|J_ g k^1_ z\|_{H^2}$ is in $L^p(\Bn,d\lambda_ n)$, and by \eqref{Eqp2}, this is equivalent to
\begin{equation}\label{CBP-2}
\int_{\Bn} \!\left (\int_{\Bn} \frac{|Rg(w)|^2\,(1-|w|^2)}{|1-\langle w,z\rangle|^{2n+2}}\,dv(w)\right )^{p/2}\!\!\! (1-|z|^2)^{p(n+2)/2} \,d\lambda_ n(z)<\infty.
\end{equation}
Now the well known estimate
\begin{displaymath}
|Rg(z)|^2\le C (1-|z|^2)^n \int_{\Bn}\frac{|Rg(w)|^2\,(1-|w|^2)}{|1-\langle w,z\rangle|^{2n+2}}\,dv(w)
\end{displaymath}
shows that \eqref{CBp} holds. The proof of the Theorem for $n\ge 2$ is now completed since for $p=n$ the condition \eqref{CBp} implies that $g$ must be constant, and $S_ p(H^2)\subset S_ n(H^2)$ for $p<n$.

 This result is a typical example of when the one dimensional case presents more difficulties, mainly because there is more work to do when $n=1$ since the case $1\le p<2$ is still not proved. By Corollary \ref{Cor1}, the condition \eqref{CBP-2} is a sufficient condition for $J_ g$ to be in $S_ p(H^2)$ when $p<2$, but this condition is easily implied by \eqref{CBp} due to H\"{o}lder's inequality and Lemma \ref{LA}. The necessity of \eqref{CBp} when $1\le p<2$ can be done as follows: if $J_ g$ is in $S_ p(H^2)$ then admits the decomposition $J_ g f=\sum_ k \lambda_ k \langle f,e_ k \rangle _{H^2} e_ k$, where $\{\lambda_ k\}$ are the singular numbers of $J_ g$ and $\{e_ k\}$ is an orthonormal set in $H^2$. By testing the previous formula on reproducing kernels and taking radial derivatives one gets
$ K_ z(w) \,R g(w)=\sum_ k \lambda_ k \,\overline{e_ k(z)}\, Re_ k(w).$
Differentiating then in $z$ and taking $w=z$ one obtains
\[ \overline{RK_ z(z)} \,R g(z)=\sum_ k \lambda_ k \,| Re_ k(z)|^2.\]
A calculation gives $RK_ z(z)=|z|^2\,(1-|z|^2)^{-2}$. Then
\begin{displaymath}
\int_{\mathbb{B}_ 1} \!\!|Rg(z)|^p \,(1-|z|^2)^p\,d\lambda_ 1(z)\le \!\int_{\mathbb{B}_ 1} \!\!\left (\sum_ k |\lambda_ k|\,|Re_ k(z)|^2 \!\right )^p \!\!|z|^{-2p}\,(1-|z|^2)^{3p}\,d\lambda_ 1(z).
\end{displaymath}
Now, H\"{o}lder's inequality yields
\begin{displaymath}
\begin{split}
\left (\sum_ k |\lambda_ k|\,|Re_ k(z)|^2 \!\right )^p &\le \left (\sum_ k |\lambda_ k|^p \,|Re_ k(z)|^2 \right ) \left (\sum_ k |Re_ k(z)|^2\right )^{p-1}
\\
& \le \left (\sum_ k |\lambda_ k|^p \,|Re_ k(z)|^2 \right ) \,\|RK_ z\|_{H^2}^{2p-2},
\end{split}
\end{displaymath}
and, since $\|RK_ z\|_{H^2}^2 \lesssim |z|^2\,(1-|z|^2)^{-3}$, we finally obtain
\begin{displaymath}
\begin{split}
\int_{\mathbb{B}_ 1} \!\!|Rg(z)|^p \,(1-|z|^2)^p\,d\lambda_ 1(z)&\le \sum_ k|\lambda_ k|^p \int_{\mathbb{B}_ 1} |Re_ k(z)|^2  \,(1-|z|^2)\,|z|^{-2}\,dv(z)
\\
& \lesssim \sum_ k|\lambda_ k|^p<\infty
\end{split}
\end{displaymath}
proving that \eqref{CBp} holds. Again, if $p=1$ then \eqref{CBp} implies that $g$ must be a constant completing the proof of the theorem.
\end{proof}
The proof of the case $n=1$ of Theorem \ref{TSCl} given in \cite{AS0} relies on the observation that $J_ g^* J_ g$ is essentially the Toeplitz type operator $Q_{\mu_ g}$ with $d\mu_ g(z)=|Rg(z)|^2\,(1-|z|^2)\,dv(z)$, and then appealing to a result of Luecking \cite{Lu} that describes, when $n=1$, the membership in the Schatten classes $S_ p(H^2)$ of the Toeplitz type operator $Q_ {\mu}$ for a positive Borel measure $\mu$ on $\Bn$, defined as
\begin{displaymath}
Q_{\mu} f(z)=\int_{\Bn} \!\!\frac{f(w)}{(1-\langle z,w \rangle )^n}\,d\mu(w),\qquad z\in \Bn.
\end{displaymath}
As far as I know, it seems that the operator $Q_{\mu}$ has not been studied in the setting of Hardy spaces in the unit ball. Here I am going to make some comments on the boundedness, compactness and membership in the Schatten ideals of the operator $Q_{\mu}$ acting on $H^2(\Bn)$, but since this is not the main topic of the paper we will not enter into the details. By using the identity
$$\langle Q_{\mu} f,g\rangle _{H^2}=\int_{\Bn} \!\!f(w)\,\overline{g(w)}\,d\mu(w) $$
it is easy to prove that $Q_{\mu}$ is bounded on $H^2$ if and only if $\mu$ is a Carleson measure, and that the compactness is characterized by $\mu$ being a vanishing Carleson measure.  Concerning the membership of $Q_{\mu}$ in the Schatten classes, Lemma \ref{KLS} can be of some help in order to prove some parts of the analogue of Luecking's result for $n>1$.

\section{Proof of Theorem \ref{LT}}\label{s7}
\subsection{Sufficiency}

Assume first that the function $\widetilde{\mu}$ belongs to $L^{p/(p-s)}(\Sn)$, and let $f\in H^p$. Then, by \eqref{EqG}, H\"{o}lder's inequality with exponent $p/s>1$ and Theorem \ref{NTMT}, we obtain
\begin{displaymath}
\begin{split}
\int_{\Bn} |f(z)|^s \,d\mu(z)&\asymp \int_{\Sn} \int_{\Gamma(\zeta)} |f(z)|^s (1-|z|^2)^{-n}\,d\mu(z)\,d\sigma(\zeta)
\\
& \le \int_{\Sn}|f^*(\zeta)|^s \int_{\Gamma(\zeta)} (1-|z|^2)^{-n}\,d\mu(z)\,d\sigma(\zeta)
\\
&\le C \|f\|_{H^p}^s \cdot \|\widetilde{\mu}\|_{L^{p/(p-s)}(\Sn)}.
\end{split}
\end{displaymath}
\subsection{A version for Poisson integrals}
Next we state and prove a version of Theorem \ref{LT} for invariant Poisson integrals $u_{\varphi}$ of  functions $\varphi$ in $L^p(\Sn)$ that can be of independent interest.
\begin{theorem}
Let $1<p<\infty$, $0<s<p$ and let $\mu$ be a finite positive Borel measure on $\Bn$. Then
\[\int_{\Bn} |u_{\varphi}(z)|^s d\mu(z) \le K_{\mu} \,\|\varphi\|_{L^p(\Sn)}^s\]
if and only if $\widetilde{\mu}\in L^{p/(p-s)}(\Sn)$. Moreover, $\|\widetilde{\mu}\|_{L^{p/(p-s)}}\asymp K_{\mu}$.
\end{theorem}
\begin{proof}
The sufficiency of the condition $\widetilde{\mu}\in L^{p/(p-s)}(\Sn)$ follows from the previous argument taking into account that $\|u^*_{\varphi}\|_{L^q(\Sn)}\le C \|\varphi\|_{L^q(\Sn)}$ for $q>1$. The proof of the necessity can be done as follows. Since $p/(p-s)>1$, then by duality,
\begin{displaymath}
\|\widetilde{\mu}\|_{L^{p/(p-s)}(\Sn)}=\sup_{\varphi} \int_{\Sn} \widetilde{\mu}(\zeta)\,\varphi (\zeta)\,d\sigma(\zeta),
\end{displaymath}
where the supremum is taken over all positive $\varphi$ in $L^{p/s}(\Sn)$ with norm one. Using the definition of $\widetilde{\mu}$, that $(1-|z|^2)\asymp |1-\langle z,\zeta \rangle|$ for $z\in \Gamma(\zeta)$, and interchanging the order of integration we arrive at
\begin{displaymath}
\int_{\Sn}\! \widetilde{\mu}(\zeta)\,\varphi (\zeta)\,d\sigma(\zeta)\asymp\!\int_{\Bn} \!\int_{I(z)}\!  \frac{\varphi(\zeta)\,(1-|z|^2)^n}{|1-\langle z,\zeta\rangle |^{2n}} \,d\sigma (\zeta)\, d\mu(z)\le \int_{\Bn}\!\!u_{\varphi}(z) d\mu(z).
\end{displaymath}
If $s=1$, this gives $\|\widetilde{\mu}\|_{L^{p/(p-s)}(\Sn)}\le C\, K_{\mu}$.
If $0<s<1$ let $f=\varphi^{1/s}\in L^p(\Sn)$. By H\"{o}lder's inequality one has $u_{\varphi}(z)\le
 u_ f(z)^s.$
Hence
$$\int_{\Sn} \widetilde{\mu}(\zeta)\,\varphi (\zeta)\,d\sigma(\zeta)\lesssim \int_{\Bn} \! u_ f(z)^s \,d\mu(z) \le K_{\mu} \, \|f\|^s_{L^p(\Sn)}=K_{\mu}\,\|\varphi\|_{L^{p/s}(\Sn)}.$$
Finally, consider the case $s>1$.  Take $t>1$ with $t<(p-1)/(s-1)$, and let $t'$ denote the conjugate exponent of $t$. By H\"{o}lder's inequality,
$ u_{\varphi}(z)\le u_ f(z)^{1/t'} \cdot u_ g(z)^{1/t}, $
with $f=\varphi^{1/s}\in L^p(\Sn)$, $g=\varphi^ {\,\sigma /s}\in L^{p/\sigma}(\Sn)$ and $\sigma=1+(s-1)t$. Another application of H\"{o}lder's inequality yields
\begin{displaymath}
\begin{split}
\int_{\Sn} \widetilde{\mu}(\zeta)\,\varphi (\zeta)\,d\sigma(\zeta)& \le \int_{\Bn} u_ f(z)^{1/t'} \cdot u_ g(z)^{1/t}\, d\mu(z)
\\
&\le \left ( \int_{\Bn} \!u_ f(z)^s d\mu(z)\right )^{\frac{1}{t's}}\!\left ( \int_{\Bn} \!u_ g(z)^{s /\sigma} d\mu(z)\right )^{\frac{\sigma}{ts}}.
\end{split}
\end{displaymath}
By our assumption, we have
\[  \int_{\Bn} \!u_ f(z)^s d\mu(z)\le K_ {\mu} \|f\|_{L^p(\Sn)}^{s}=K_ {\mu} \|\varphi\|_{L^{p/s}(\Sn)}.\]
On the other hand, the choice of $t$ makes $p/\sigma >1$ and therefore, assuming that $\mu$ has compact support on $\Bn$, the proof of the sufficiency part gives
\[\int_{\Bn} \!u_ g(z)^{s/\sigma} d\mu(z) \lesssim \|\widetilde{\mu}\|_{L^{p/(p-s)}(\Sn)} \cdot \|g\|^{s/\sigma}_{L^{p/\sigma}(\Sn)}=\|\widetilde{\mu}\|_{L^{p/(p-s)}(\Sn)} \cdot \|\varphi\|_{L^{p/s}(\Sn)}.\]
All together yields
\[ \int_{\Sn} \widetilde{\mu}(\zeta)\,\varphi (\zeta)\,d\sigma(\zeta) \lesssim K_{\mu}^{1/t's}\cdot \|\widetilde{\mu}\|^{\sigma/ts}_{L^{p/(p-s)}(\Sn)} \cdot \|\varphi\|_{L^{p/s}(\Sn)}\]
proving that $\|\widetilde{\mu}\|_{L^{p/(p-s)}(\Sn)}\le C \,K_{\mu}$.
This gives the result when $\mu$ has compact support on $\Bn$. The result for arbitrary $\mu$ follows from this by an easy limit argument.
\end{proof}
\subsection{The tent spaces $T^p(Z)$}
A sequence of points $\{z_ j\}\subset \Bn$ is said to be separated if there exists $\delta>0$ such that $\beta(z_ i,z_ j)\ge \delta$ for all $i$ and $j$ with $i\neq j$, where $\beta(z,w)$ denotes the Bergman metric on $\Bn$. This implies that there is $r>0$ such that the Bergman metric balls $D_ j=\{z\in \Bn :\beta(z,z_ j)<r\}$ are pairwise disjoints. Taking into account that $v(D_ j)\asymp (1-|z_ j|^2)^{n+1}$, is then an easy consequence of Lemma \ref{LA} that, if $\{z_ j\}$ is a separated sequence in $\Bn$, for $t>n$ one has
\begin{equation}\label{Sep}
\sum_ j \frac{(1-|z_ j|^2)^t}{|1-\langle z,z_ j \rangle |^{t+\varepsilon}} \le C (1-|z|^2)^{-\varepsilon},\qquad z\in \Bn.
\end{equation}
For $0<p<\infty$ and a fixed separated sequence $Z=\{z_ j\}\subset \Bn$, let $T^p(Z)$ consist of those sequences $\lambda=\{\lambda_ j\}$ of complex numbers with
\[ \|\lambda \|_{T^p}^p =\int_{\Sn} \!\!\Big (\!\!\sum_{z_ j\in \Gamma(\zeta)} \!|\lambda_ j |^2 \Big )^{p/2} d\sigma(\zeta)  <\infty.\]
The following result can be thought as the holomorphic analogue of Lemma 3 in Luecking's paper \cite{Lue1}.
\begin{proposition}\label{TKL}
Let $Z=\{z_ j\}$ be a separated sequence in $\Bn$ and let $0<p<\infty$. If $b>n\max(1,2/p)$, then the operator $T_{Z}: T^p(Z)\rightarrow H^p$ defined by
$$T_{Z}(\{\lambda_ j\})=\sum_ j \lambda_ j \,\frac{(1-|z_ j|^2)^{b}}{(1-\langle z, z_ j  \rangle )^b}$$ is bounded.
\end{proposition}
\begin{proof}
Let $\lambda=\{\lambda_ j\}\in T^p(Z)$ and set $g(z)=T_ Z (\lambda)(z)$. By \cite{AB} it is enough to prove that $\|A^k(g)\|_{L^p(\Sn)}\le C \|\lambda\|_{T^p(Z)}$ for some positive integer $k$, where
\[ A^k(g)(\zeta)=\left (\int_{\Gamma(\zeta)} |R^k g(z)|^2 \,(1-|z|^2)^{2k}\,d\lambda_ n(z) \right )^{1/2}.\]
Easy computations involving radial derivatives together with Cauchy-Schwarz implies
\begin{displaymath}
\begin{split}
|R^{n+1} \!g(z)|^2&\lesssim \left (\sum_ j |\lambda_ j| \frac{(1-|z_ j|^2)^{b}}{|1-\langle z,z_ j \rangle|^{b+n+1}}\right )^2
\\
&\le  \left (\sum_ j |\lambda_ j |^2 \frac{(1-|z_ j|^2)^{b}}{|1-\langle z,z_ j \rangle|^{b+n+1}}\right ) \left (\sum_ j \frac{(1-|z_ j|^2)^{b}}
{|1-\langle z,z_ j \rangle|^{b+n+1}}\right ).
\end{split}
\end{displaymath}
This together with \eqref{Sep} gives
\begin{displaymath}
\begin{split}
\big (A^{n+1}(g)(\zeta)\big )^2&\lesssim \sum_ j |\lambda_ j |^2 (1-|z_ j|^2)^b
 \int_{\Gamma(\zeta)}\frac{dv(z)}{|1-\langle z,z_ j \rangle|^{b+n+1}}
\\
& \lesssim \sum_ j |\lambda_ j |^2 \frac{(1-|z_ j|^2)^b} {|1-\langle \zeta,z_ j \rangle|^{b}}.
\end{split}
\end{displaymath}
In the last estimate it has been used that, since $(1-|z|^2)\asymp |1-\langle z,\zeta \rangle |$ for $z\in \Gamma(\zeta)$, then due to \cite[Lemma 2.5]{OF} one has
\begin{displaymath}
\begin{split}
\int_{\Gamma(\zeta)}\!\frac{dv(z)}{|1-\langle z,z_ j \rangle |^{b+n+1}}&\asymp \int_{\Gamma(\zeta)}\!\frac{(1-|z|^2)^{b+1} dv(z)}{|1-\langle z,\zeta \rangle |^{b+n+1}|1-\langle z,z_ j \rangle|^{b+1}}
\\
& \lesssim |1-\langle \zeta,z_ j\rangle |^{-b}.
\end{split}
\end{displaymath}
Therefore,
\begin{displaymath}
\begin{split}
\|A^{n+1} \!g\|^p_{L^p(\Sn)}\lesssim \int_{\Sn}\left ( \sum_ j |\lambda_ j |^2 \frac{(1-|z_ j|^2)^b} {|1-\langle \zeta,z_ j\rangle |^{b}}\right )^{p/2}\!d\sigma(\zeta)
\end{split}
\end{displaymath}
and the proof is finished after the use of the estimate, valid for $s>0$ and $b>n\max(1,1/s)$,
\begin{displaymath}
\int_{\Sn} \left ( \int_{\Bn} \Big (\frac{1-|z|^2}{|1-\langle z,\zeta\rangle|}\Big )^b d\mu(z)\right )^s \,d\sigma(\zeta)\le C \int_{\Sn} \mu(\Gamma(\zeta))^s\,d\sigma(\zeta),
\end{displaymath}
with $\mu$ being a positive measure on $\Bn$. This estimate is the analogue of Proposition 1 in Luecking's paper \cite{Lue1} and is proved in the same way.
\end{proof}

\subsection{Necessity}
We follow the argument of Luecking \cite{Lue1}. According to \cite{Xia}, for each positive integer $k\ge 20$, there are points $\{\zeta_{jk}\}_{j=1}^{m(k)}\subset \Sn$ such that $\Sn=\displaystyle{\bigcup_{j=1}^{m(k)}Q(\zeta_{jk},2^{-k})}$ and
\begin{equation}\label{Nec-1}
Q\big(\zeta_{ik},\frac{1}{9}2^{-k}\big )\cap Q\big (\zeta_{jk},\frac{1}{9}2^{-k}\big )=\emptyset \quad \textrm{ if }\quad i\neq j.
 \end{equation}
Recall that $Q(\zeta,\delta)=\{ \xi\in \Sn : |1-\langle \zeta, \xi \rangle |<\delta\}$. We denote by $\mathcal{N}_ k$ the collection of all non-isotropic balls $Q(\zeta_{jk},2^{-k})$, $1\le j\le m(k)$, and let $\mathcal{N}=\bigcup \mathcal{N}_ k$.  Also, any point $\zeta\in \Sn$ belongs to at most $N$ balls in $\mathcal{N}_ k$, where $N$ depends only on the dimension. If $Q=Q(\zeta,\delta)$ we use the notation $\widehat{Q}=B_{\delta}(\zeta)=\{z\in \Bn: |1-\langle z,\zeta \rangle |<\delta\}$. As in \cite{Lue1}, it is enough to show that the function
\begin{displaymath}
\zeta \mapsto \sup \left \{ \frac{\mu(\widehat{Q})}{\sigma(Q)}:Q\in \mathcal{N}, \zeta\in Q \right \}
\end{displaymath}
belongs to $L^{p/(p-s)}(\Sn)$. Thus, we may assume that $\widetilde{\mu}$ is the above supremum. For each positive integer $m$, let $\mathcal{E}_ m$ denote the collection of all ``maximal" balls $Q\in \mathcal{N}$ with $\mu(\widehat{Q})>2^m \sigma (Q)$, and set $\mathcal{E}=\bigcup \mathcal{E}_ m$. The construction of $\mathcal{E}_ m$ goes as follows: for a fixed $k_ 0$, let $\mathcal{G}_ 0^m$ be the collection of all balls $Q\in \mathcal{N}_{k_ 0}$ with $\mu(\widehat{Q})>2^m \sigma (Q)$. Once $\mathcal{G}^m_ {\ell-1}$ is constructed, then $\mathcal{G}^m_ {\ell}$ consists of those balls $Q\in \mathcal{N}_{k_ 0+\ell}$ satisfying $\mu(\widehat{Q})>2^m \sigma (Q)$ such that $Q$ is not contained in any ball in $\bigcup_{i=0}^{\ell-1}  \mathcal{G}^m_ i$, and then $\mathcal{E}_ m=\bigcup _{i\ge 0}\mathcal{G}^m_ i$. With this construction, is clear that $E_{m+1}\subset E_ m$, where $E_ m=\bigcup _{Q\in \mathcal{E}_ m} Q$. Also, if $Q_ 1=Q(\zeta_ 1,\delta_ 1)$ and $Q_ 2=Q(\zeta_ 2,\delta_ 2)$ are two distinct balls in $\mathcal{E}_ m$, then $$Q\big(\zeta_ 1,\frac{1}{81}\delta_ 1 \big ) \cap Q \big(\zeta_ 2,\frac{1}{81}\delta_ 2\big )=\emptyset.$$
If $Q_ 1$ and $Q_ 2$ are in the same generation, this follows from \eqref{Nec-1}; and if they belong to different generations and the previous intersection is not empty, then one ball is strictly included in the other and therefore would not have been picked.

If $Q=Q(\zeta,\delta)\in \mathcal{E}$, let $z_ Q=(1-c(\alpha)\delta)\,\zeta$ with $c(\alpha)=\big (81\cdot 4\alpha)^{-1}$. Recall that $\alpha$ is the aperture of the admissible approach regions. It is not hard to verify that $Z=\{z_ Q:Q\in \mathcal{E}\}$ is a separated sequence.  By taking $\mu$ with compact support on $\Bn$, we may assume that $Z$ is a finite sequence.

Now, for $b>n\max(1,2/p)$ and $\lambda=\{\lambda_{Q}:Q\in \mathcal{E}\}\in T^p(Z)$, consider the function
\[ f_ t(z)=\sum_{Q\in \mathcal{E}} \lambda_ Q \,r_{Q}(t) \,\frac{(1-|z_ Q|^2)^{b}}{(1-\langle z, z_ Q  \rangle )^b},\quad z\in \Bn,\quad 0<t<1,\]
where $r_ Q(t)$ is a sequence of Rademacher functions. Using our assumption, Proposition \ref{TKL}, integrating on $t$ and applying Khinchine's inequality we get
\begin{displaymath}
\int_{\Bn} \!\!\Big ( \sum _{Q\in \mathcal{E}} |\lambda_ Q|^2 \,F_ Q(z)^{2b}\Big )^{s/2} d\mu(z)\le C \|I_ d\|_{H^p\rightarrow L^s(\mu)}^s \cdot \|\lambda\|^s_{T^p}
\end{displaymath}
with $F_ Q(z)=(1-|z_ Q|^2)/|1-\langle z,z_ Q\rangle|$ that satisfies $F_ Q(z)\ge C>0$ for $z\in \widehat{Q}$. Set $\widehat{E}_ m=\bigcup_{Q\in \mathcal{E}_ m} \widehat{Q}$, and for $Q\in \mathcal{G}^m_ {\ell}$ let
$$G(Q)=\widehat{Q} \,\setminus \,\widehat{Q}\cap \widehat{E}_{m+1}\,\setminus \,\bigcup \big \{\widehat{Q}\cap \widehat{Q'} : Q'\in \mathcal{G}^m_{i},\,i>\ell\big \}.$$
It is obvious that $G(Q_ 1)\cap G(Q_ 2)=\emptyset$ if $Q_ 1$ and $Q_ 2$ belong to distinct  $\mathcal{E}_ m$, and this continues to hold if they are in different generations of the same $\mathcal{E}_ m$. Thus, any point $z\in \Bn$ belongs to at most $N$ sets $G(Q)$ with $N$ depending only on the dimension. It follows that $$\Big ( \sum _{Q\in \mathcal{E}} |\lambda_ Q|^2 \,\chi_{\widehat{Q}}(z)\Big )^{s/2}\ge C\sum _{Q\in \mathcal{E}} |\lambda_ Q|^s \,\chi_{G(Q)}(z),$$ with $C=\min(1,N^{\frac{s-2}{2}})$.
Therefore, we obtain
\begin{equation}\label{Eq7}
\sum _{Q\in \mathcal{E}} |\lambda_ Q|^s \,\mu(G(Q))\le C \|I_ d\|_{H^p\rightarrow L^s(\mu)}^s \cdot \|\lambda\|^s_{T^p}.
\end{equation}
We will apply this inequality to an appropriate sequence of numbers $\{\lambda_ Q\}$. Put $r=p/(p-s)$ and set $\lambda_{Q}=2^{\frac{m}{s}(r-1)}$ if $Q\in \mathcal{E}_ m$. Notice that
\begin{displaymath}
\sum_{Q\in \mathcal{E}_ m} \mu \big (G(Q)\big ) \ge \mu \Big ( \bigcup _{Q\in \mathcal{E}_ m}\! G(Q) \Big )=\mu \big(\widehat{E}_ m \setminus \widehat{E}_{m+1}\big).
\end{displaymath}
Then
\begin{displaymath}
\begin{split}
\sum_{Q\in \mathcal{E}} |\lambda_ Q|^s \,\mu(G(Q))&=\sum_ m 2^{m(r-1)} \sum_{Q\in \mathcal{E}_ m} \mu \big (G(Q)\big )
\\
&\ge \sum_ m 2^{m(r-1)} \Big (\mu(\widehat{E}_ m)- \mu ( \widehat{E}_{m+1})\Big ).
\end{split}
\end{displaymath}
By a typical covering lemma of Vitali type (see \cite[p.\,9]{Stein}), there is a sequence $\mathcal{F}_ m$ of pairwise disjoint balls $Q\in \mathcal{E}_ m$ with
$\sigma (E_ m) \le C \sum_{Q\in \mathcal{F}_ m} \sigma(Q)$ (here the constant $C$ depends only on the dimension). This implies
\begin{displaymath}
\begin{split}
\mu(\widehat{E}_ m)&=\mu \Big ( \!\bigcup_{Q\in \mathcal{E}_ m} \!\!\widehat{Q} \Big ) \ge \mu \big ( \!\!\bigcup_{Q\in \mathcal{F}_ m} \!\!\widehat{Q} \big )=\!\sum _{Q\in \mathcal{F}_ m} \!\mu(\widehat{Q})
\\
&\ge 2^m \!\!\sum _{Q\in \mathcal{F}_ m} \!\sigma(Q)\ge C 2^m \sigma(E_ m).
\end{split}
\end{displaymath}
Now, summing by parts  we obtain
\begin{equation}\label{Est1}
\begin{split}
\sum _{Q\in \mathcal{E}}|\lambda_ Q|^s \,\mu(G(Q))&
\ge \sum_ m \big (2^{m(r-1)}-2^{(m-1)(r-1)} \big ) \,\mu(\widehat{E}_ m)
\\
&\ge C \sum_ m   2^{mr}\, \sigma(E_ m)\ge C \,\|\widetilde{\mu}\|_{L^r(\Sn)}^r,
\end{split}
\end{equation}
where  the last estimate is due to the fact that $\widetilde{\mu}(\zeta) \asymp 2^m$ for $\zeta \in E_ m\setminus E_{m+1}$. On the other hand, using that $d(z,w)=|1-\langle z,w \rangle |^{1/2}$ satisfies the triangle inequality \cite[Proposition 5.1.2]{Rud} together with the choice made on the points $z_ Q$, we see that $z_ Q\in \Gamma(\zeta)$ implies that $\zeta \in \widetilde{Q}$, where $\widetilde{Q}=Q(\xi, \frac{1}{81} \delta)$ if $Q=Q(\xi,\delta)$. We know that $\widetilde{Q}_ 1 \cap \widetilde{Q}_ 2=\emptyset$ if $Q_ 1$ and $Q_ 2$ are in $\mathcal{E}_ m$. Therefore,
\begin{displaymath}
\begin{split}
\|\lambda \|^p_{T^p}& =\int_{\Sn} \!\!\Big (\!\sum_{z_ Q\in \Gamma(\zeta)} |\lambda _ Q|^2 \Big )^{p/2} \!\! d\sigma(\zeta)
\le \int_{\Sn} \!\!\Big (\sum _{Q\in \mathcal{E}}|\lambda _ Q|^2 \chi _{{\widetilde{Q}}}(\zeta)\Big )^{p/2} \!\! d\sigma(\zeta)
\\
&\le \int_{\Sn} \!\!\Big (\sum_ {m}  2^{\frac{2m}{s}(r-1)} \chi_{_{E_ m}}(\zeta)\Big )^{p/2} \!\! d\sigma(\zeta).
\end{split}
\end{displaymath}
Finally, a summation by parts gives
\begin{equation}\label{Est2}
\begin{split}
\|\lambda \|^p_{T^p}&\le C\int_{\Sn} \!\!\Big (\sum_ {m}  2^{\frac{2m}{s}(r-1)} \chi_{_{E_ m\setminus E_{m+1}}}(\zeta)\Big )^{p/2} \!\! d\sigma(\zeta)
\\
&=C \sum_ m  2^{\frac{mp}{s}(r-1)}\sigma(E_ m\setminus E_{m+1})\le C \|\widetilde{\mu}\|^r_{L^r(\Sn)}.
\end{split}
\end{equation}
Putting the two previous estimates \eqref{Est1} and \eqref{Est2} into \eqref{Eq7} gives
\[ \|\widetilde{\mu}\|^r_{L^r(\Sn)} \le C \|I_ d\|_{H^p\rightarrow L^s(\mu)}^s \cdot \|\widetilde{\mu}\|^{rs/p}_{L^r(\Sn)} \]
that gives $\|\widetilde{\mu}\|_{L^r(\Sn)} \le C \|I_ d\|_{H^p\rightarrow L^s(\mu)}^s$ for $\mu$ with compact support on $\Bn$. The result for an arbitrary measure $\mu$ follows from this by an standard limit argument. This completes the proof of Theorem \ref{LT}.

\end{document}